\documentclass[12pt]{amsart}
\usepackage{ifthen,amssymb,graphicx}
\usepackage[all]{xy}
\usepackage[usenames]{color}
\usepackage{mathrsfs}

\newcommand{\isom}{\stackrel{\sim}\to}

\newcommand{\catD}{{D}}
\newcommand{\dbc}[1]{{D}^b(#1)}
\newcommand{\dpc}[1]{{D}^+(#1)}
\newcommand{\dmc}[1]{{D}^-(#1)}
\newcommand{\cdbc}[1]{{D}^b_c(#1)}

\newcommand{\fmf}[3]{{\Phi^{#1}_{{\scriptscriptstyle #2\!\rightarrow\! #3}}}}
\newcommand{\Hom}{{\operatorname{Hom}}}
\newcommand{\Aut}{{\operatorname{Aut}}}
\newcommand{\FM}{{\operatorname{FM}}}
\newcommand{\NS}{{\operatorname{NS}}}
\newcommand{\vb}{{\operatorname{Vec}}}

\newcommand{\End}{{\operatorname{End}}}

\newcommand{\SL}{{\operatorname{SL}}}

\newcommand{\lotimes}{{\,\stackrel{\mathbf L}{\otimes}\,}}

\DeclareMathOperator{\Pic}{{Pic}}

\newcommand{\bbZ}{{\mathbb Z}}
\newcommand{\bbP}{{\mathbb P}}

\newcommand{\cE}{{\mathcal E}}

\newcommand{\cI}{{\mathcal I}}

\newcommand{\calL}{{\mathcal L}}
\newcommand{\cK}{{\mathcal K}}
\newcommand{\cM}{{\mathcal M}}

\newcommand{\cO}{{\mathcal O}}
\newcommand{\cP}{{\mathcal P}}

\newcommand{\bR}{{\mathbf R}}
\newcommand{\bL}{{\mathbf L}}

\newcommand{\cplx}[1]{{{\mathcal #1}^{\scriptscriptstyle\bullet}}}

\newcommand{\dSHom}[1]{{\mathcal{H}om_{#1}^{\scriptscriptstyle\bullet}}}
\newcommand{\rcplx}[2]{{{\mathcal #1}_{#2}^{\scriptscriptstyle\bullet}}}

\newcommand{\rk}{\operatorname{rk}}
\newcommand{\ch}{\operatorname{ch}}

\newcommand{\iso}{{\,\stackrel {\textstyle\sim}{\to}\,}}
\newcommand{\lra}{\longrightarrow}
\newcommand{\hra}{\hookrightarrow}

\newtheorem{thm}{Theorem}[section]
\newtheorem*{thm*}{Theorem}
 \newtheorem*{prop*}{Proposition}
\newtheorem{cor}[thm]{Corollary}
\newtheorem{lem}[thm]{Lemma}
\newtheorem{prop}[thm]{Proposition}

\theoremstyle{definition}
\newtheorem{defin}[thm]{Definition}
\theoremstyle{remark}
\newtheorem{rema}[thm]{Remark}
\newtheorem{exe}[thm]{Example}

\newenvironment{rem}{\begin{rema}}{\hfill\hspace{1pt}$\triangle$\end{rema}}

\numberwithin{equation}{section} 

\begin{document}
\title[Relative Fourier-Mukai transforms on fibrations]{Relative Fourier-Mukai transforms for Weierstra\ss\ fibrations, abelian schemes  and Fano fibrations}

\date{\today}
\thanks {Work supported by research projects MTM2009-07289 (MEC)
and GR.46/2008 (JCYL).}

\author[A. C. L\'opez Mart\'{\i}n]{Ana Cristina L\'opez Mart\'{\i}n}
\author[D. S\'anchez G\'omez]{Dar\'{\i}o S\'anchez G\'omez}
\author[C. Tejero Prieto]{Carlos Tejero Prieto}
\email{anacris@usal.es, dario@usal.es, carlost@usal.es}
\address{Departamento de Matem\'aticas and Instituto Universitario de F\'{\i}sica Fundamental y Matem\'aticas
(IUFFyM), Universidad de Salamanca, Plaza de la Merced 1-4, 37008
Salamanca, Spain.}

\begin{abstract}  
We study the group of relative Fourier-Mukai transforms for Weierstra\ss\ fibrations, abelian schemes and Fano or anti-Fano fibrations. For Weierstra\ss\ and Fano or anti-Fano fibrations  we  describe this group completely. For abelian schemes over an arbitrary base we prove that if two of them are relative Fourier-Mukai partners then there is an isometric isomorphism between the fibre products of each of them and its dual abelian scheme. If the base is normal and the slope map is surjective we show that these two conditions are equivalent.  Moreover in this situation we completely determine the group of relative Fourier-Mukai transforms and we prove that the number of relative Fourier-Mukai partners of a given abelian scheme over a normal base is finite. \end{abstract}

\maketitle
\setcounter{tocdepth}{1}

{\small \tableofcontents }

\section*{Introduction}
In this paper we endeavour to determine, for several fibrations $X\to T$, the group $\FM_T(\cdbc{X})$ of relative Fourier-Mukai transforms (i.e autoequivalences of the bounded derived category of $X$ representable by an object $\cplx{K}\in \cdbc{X\times_T X}$ supported on the fibered product $X\times_TX$).

In the absolute case, that is, when the base scheme $T$ is the spectrum of a field $\Bbbk$, if $X$ is projective, thanks to Orlov's representation theorem  \cite{Or97} as extended by Ballard \cite{Ballard09}, any (exact) autoequivalence of the bounded derived
category of $X$  is an integral functor and thus the group $\FM_\Bbbk(\cdbc{X})$ coincides  with the whole group $\Aut_\Bbbk\cdbc{X}$ of autoequivalences. The determination of this group is widely recognized as an important problem and has been considered by several authors (see \cite{BO02},  \cite{HiVdB}, \cite{Pol96}, \cite{Or02}, \cite{HMS09} and  \cite{BroPloog10}).


In the relative setting of a fibration $X\to T$,  the analogous role of the group of autoequivalences is played by the group $\Aut_T\cdbc{X}$  of $T$-linear autoequivalences of the bounded derived category of $X$, a  notion introduced for the first time by Kuznetsov in \cite{Kuz07}.  Projection formula in the derived category implies that $\FM_T(\cdbc{X})$  is a subgroup of  $\Aut_T\cdbc{X}$, but in contrast to the absolute case,  it is still an open problem to prove if both groups are equal or not (see \cite{CanStella11} for a recent survey on the subject). 

The characterization of all $T$-linear autoequivalences seems to be a hard problem. As a first step in this direction one might try to determine the subgroup $\FM_T(\cdbc{X})$ given by relative Fourier-Mukai transforms. Notice moreover that this subgroup  is  interesting on its own since compatibility of relative Fourier-Mukai transforms with base change make them a powerful tool with important  applications in Geometry and Physics. Let us cite here just three of them. The first one is that they allow one to establish, in a very natural way, algebraic isomorphisms between relative moduli spaces of sheaves, see \cite[Section 6.4.2]{BBH08}, \cite{HLST09}. The second is that one can use them to construct stable vector bundles on fibered varieties, \cite{HMP02, AH03, AHS10}. The last one is that they serve to formulate in mathematical terms what physicists call fibrewise  $T$-duality for $D$-branes in mirror symmetry \cite{AYCH01}.  However, to the best of our knowledge, up to now nobody has determined the group of relative Fourier-Mukai transforms $\FM_T(\cdbc{X})$ for any fibration. 
We propose to remedy here this situation.The main contribution of this paper is the description of this group for Weierstra\ss\ fibrations, abelian schemes and Fano or anti-Fano fibrations as an extension of groups. The group $\FM_T(\cdbc{X})$ has the natural subgroup $\FM_T^0(\cdbc{X})=\Aut(X/T)\rtimes(\bbZ\times  \Pic(X)) $ generated by the so called trivial autoequivalences: relative automorphisms, shifts of complexes and twists by line bundles. Therefore one might try to describe the group of relative Fourier-Mukai transforms as an extension whose kernel is this natural subgroup. However this is not always possible, since, in general,  $\FM_T^0(\cdbc{X})$ is not a normal subgroup of  $\FM_T(\cdbc{X})$. Hence, the best we can ask for is that the kernel of the extension be a suitable subgroup of $\FM_T^0(\cdbc{X})$ normal inside $\FM_T(\cdbc{X})$.

Let us explain now our strategy for achieving this goal. In first place we make use of a very useful result proved recently in a paper coauthored by one of the authors  \cite[Proposition 2.15]{HLS08} namely that a relative integral functor between the derived category of locally projective fibrations is a relative Fourier-Mukai transform if and only if its restriction to every fibre is a Fourier-Mukai transform. Therefore we have a group morphism $$\alpha\colon\FM_T(\cdbc{X})\to\prod_{t\in T}\FM(\cdbc{X_t})\, .$$
Nevertheless,  notice that even  with this result at hand and knowing the precise descrip\-tion of the group of Fourier-Mukai transforms of the fibers, it is not  a trivial problem to describe the group of relative Fourier-Mukai transforms for the fibration since $\alpha$ is not surjective. Thus, in order to be able to describe this group we need more information regarding the behavior of restrictions of kernels to fibers. We get the additional information from two key results. The first one is
 
 \begin{prop*}[Proposition \ref{prop: ncte}] Let $\cplx{K}$ be an object in $\cdbc{X\times_ T Y}$. Suppose $X$ is connected and that for every closed point $t\in T$, the restriction $\bL j_t^\ast \cplx{K}\simeq \mathcal{K}_t[n_t]$ where $\cK_t$ is a sheaf on $X_t\times Y_t$ flat over $X_t$ and $n_t\in \mathbb{Z}$. Then, $\cplx{K}\simeq \mathcal{K}[n]$ for some  sheaf $\mathcal{K}$ on $X\times_T Y$ flat over $X$ and  some $n\in \mathbb{Z}$.
 \end{prop*}

We prove that the kernels for relative Fourier-Mukai transforms for the fibrations considered in this paper satisfy  the conditions of this Proposition.  For Weierstra\ss\ fibrations this follows from the fact that all Fourier-Mukai kernels of the fibres are shifted universal sheaves, see Proposition \ref{prop:flatkernel}. In the case of abelian schemes we prove it in Proposition  \ref{flatness}. For Fano and anti-Fano fibrations this is immediate since in this case all Fourier-Mukai transforms of the fibres are trivial.

The second key result that we need is the generalization to the relative setting of the well known fact which states that a integral functor that sends skyscraper sheaves to skyscraper sheaves is a trivial equivalence:

\begin{prop*}[Proposition  \ref{p:trivial}] Let $\cplx{K}$ be an object in $\cdbc{X\times_ T Y}$ of finite homological dimension over $X$ and let $\fmf{\cplx{K}}{X}{Y}$ be the corresponding integral functor. Assume that for every closed point $x\in X$ one has $\fmf{\cplx{K}}{X}{Y}(\mathcal{O}_x)\simeq \mathcal{O}_y$ for some closed point $y\in Y$. Then, there exists a  relative morphism $f\colon X\to Y$ and a line bundle $\mathcal{L}$ on $X$ such that $\fmf{\cplx{K}}{X}{Y}(\cplx{E})\simeq \bR f_\ast(\mathcal{L}\otimes\cplx{E})$ for any object $\cplx{E}$ in $\cdbc{X}$.
\end{prop*}

By means of these two key results we obtain the following descriptions of the groups of relative Fourier-Mukai transforms:

\begin{thm*}[Theorem \ref{t:FMW}] Let $X\to T$ be a Weierstra\ss \  fibration over a connected base scheme.
There exists an  exact sequence
$$1\lra \Aut^0_T\cdbc{X}\lra \FM_T(\cdbc{X})\stackrel{\widetilde{\ch}}\lra \SL_2(\bbZ)\lra 1\, ,$$
with $\Aut^0_T\cdbc{X}=\Aut(X/T)\rtimes(2\bbZ\times \Pic^{\underline 0}(X)) $, where $$\Pic^{\underline 0}(X)=\{\calL\in \Pic(X) ;  \deg(\calL_t)=0 \text{, for any } t\in T\}.$$
\end{thm*}

For abelian schemes, among other results, we prove 

\begin{thm*}[Theorem \ref{t:FMabelianoNormal}] Let $X\to T$ be an abelian scheme over a connected normal base such that the slope map $\mu\colon \vb^{sh}(X\times_T X/T) \to  \NS(X\times_T X/T)\otimes_\mathbb{Z}\mathbb{Q}$ is surjective. Then there is a short exact sequence of groups:
$$0\to \mathbb{Z}\oplus (X(T)\times \hat{X}(T)\times \Pic(T))\to \FM_T(\cdbc{X})\to U(X\times_T\hat{X})\to 1\, .$$
\end{thm*}

Finally for  Fano and anti-Fano fibrations, we get

\begin{thm*}[Theorem \ref{t:Fano}] Let $X\to T$ be a Fano or anti-Fano fibration.Then,
the group $\FM_T(\cdbc{X})$ of relative Fourier-Mukai transforms is generated by relative automorphisms of $X$, twists by line bundles and shifts. That is
$$\FM_T(\cdbc{X})\simeq \Aut (X/T)\ltimes (\Pic(X)\oplus \bbZ)\, .$$
\end{thm*}

Let us now explain in more detail the contents of the paper and its organization.

In Section 1 we briefly recall some facts about the theory of relative integral functors and relative Fourier-Mukai transforms. This section also contains the key  results mentioned before and that we use all along the article.
In Section 2,  we give the description (Theorem \ref{t:FMW}) of the group of relative Fourier-Mukai transforms for any relatively integral elliptic fibration over a connected base scheme as an extension of the subgroup of even shift trivial transforms by $\SL_2(\bbZ)$. 
 In Section 3, we study the case of a projective abelian scheme,  first  over a connected base scheme  and then we pass to analyse  the particular case of a connected normal base scheme.  For  an arbitrary connected base $T$, we prove (Theorem \ref{t:FM-partners}) that if two abelian schemes $X\to T$, $Y\to T$ are relative Fourier-Mukai partners, then there is an isometric isomorphism $f\colon X\times_T\hat X \xrightarrow{\sim} Y\times_T\hat Y$. We also show (Corollary \ref{c:FMabeliano})  that the group of relative Fourier-Mukai autoequivalences forms part of a short exact sequence whose quotient term is contained in the group of isometric automorphisms. We prove that they coincide (Theorem \ref{t:abelpol}) if the abelian scheme has minimal endomorphisms and admits a line bundle which induces a polarization.   When the base is normal,  we  prove a finiteness result  (Theorem \ref{t:finitepartner}) for relative Fourier-Mukai partners. If in addition the slope map is surjective we show (Theorem \ref{t:equivalent}) that the property of being relative Fourier-Mukai partners is equivalent to the existence of an isometric isomorphism. Moreover, under the same conditions we are able to complete (Theorem \ref{t:FMabelianoNormal}) the description of the group of relative Fourier-Mukai autoequivalences, proving that it is an extension of the group of isometric isomorphisms.  Section 4  contains the description of the group of relative Fourier-Mukai transforms for Fano and anti-Fano fibrations. In this case, we prove (Theorem \ref{t:Fano}) that the group of relative Fourier-Mukai autoequivalences coincides with the trivial transforms and moreover that every $T$-linear equivalence is a relative  Fourier-Mukai transform.

\subsubsection*{Conventions}

In this paper, scheme means separated  scheme of
finite type over an algebraically closed field $\Bbbk$ of characteristic zero and unless otherwise stated a point means a closed point. By a
Gorenstein morphism, we understand a flat
morphism of schemes whose fibres are  Gorenstein. For any scheme $X$ we denote by $\catD(X)$ the
derived category of complexes of $\cO_X$-modules with
quasi-coherent cohomology sheaves. Analogously $\dpc{X}$, $\dmc{X}$
and $\dbc{X}$ denote the derived categories of complexes
which are respectively bounded below, bounded above and bounded on
both sides, and have quasi-coherent cohomology sheaves. The
subscript $c$ refers to the corresponding subcategories of
complexes with coherent cohomology sheaves. For abelian schemes we always assume that the base scheme is connected. For any relative scheme $p\colon X\to T$, we denote by $\textbf{Pic}_{X/T}$ the relative Picard functor and if it is representable then $\Pic_{X/T}$ denotes the representing scheme. We also use the notation $\Pic(X/T):=\textbf{Pic}_{X/T}(T)$.

\section{Relative Fourier-Mukai transforms}
To start with we give a short review of some basic features of the theory of relative integral functors and relative Fourier-Mukai transforms. For further details we refer to \cite{BBH08, Huybbook}.

Let $p\colon X\to T$ and $q\colon Y\to T$ be proper morphisms. We denote by $\pi_i$ the projection of the fibre product  $X\times_T Y$ onto the $i$-factor and by $\rho=p\circ\pi_1=q\circ\pi_2$ the projection of $X\times_T Y$ onto the base $T$. We have the diagram
$$\xymatrix{& X\times_T Y\ar[dl]_{\pi_1}\ar[dd]^\rho\ar[dr]^{\pi_2} &\\
                     X\ar[dr]_p & & Y\ar[dl]^q\\
	             & T &  }$$
Let $\cplx K$ be an object in $D^-(X\times_T Y)$, the {\it relative integral functor} defined by $\cplx{K}$ is the functor $\Phi^{\cplx K}_{X\to Y}\colon D^-(X)\to D^-(Y)$ given by 
$$\Phi_{X\to Y}^{\cplx K}(\cplx{E})=\bR \pi_{2_\ast}(\bL \pi_1^\ast\cplx{E}\lotimes\cplx{K})\, .$$

The complex $\cplx K$ is said to be the kernel of the integral functor. Note that any relative integral functor  can be considered as an absolute integral functor. 
One just considers the immersion $\iota\colon X\times_T Y\hra X\times Y$ and then  $\Phi^{\cplx K}_{X\to Y}$ coincides with the absolute integral functor with kernel $\iota_\ast\cplx K$.

As in the absolute situation, the composition of two relative integral functors is obtained by convolving the corresponding kernels. That is, given two kernels $\cplx{K}\in D^-(X\times_TY)$ and $\cplx{H}\in D^-(Y\times_T Z)$ corresponding to the relative integral functors $\Phi$ and $\Psi$, the kernel of the composition $\Psi\circ\Phi$ is given by the formula
\begin{equation}\label{e:compos}
\cplx{H}\ast_T\cplx{K}=\bR\pi_{XZ\ast}(\bL\pi_{XY}^\ast\cplx{K}\lotimes\bL\pi_{YZ}^\ast\cplx{H})\, , 
\end{equation}
where $\pi_{XZ},\pi_{XY}$ and $\pi_{YZ}$ are the projections from $X\times_TY\times_T Z$ onto $X\times_TZ$,  $X\times_TY$ and  $Y\times_TZ$ respectively.

In order to determine when integral functors map bounded complexes to bounded complexes, the following notion was introduced in \cite{HLS07}.
\begin{defin}
Let $f\colon Z \to T$ be a morphism of
schemes. An object $\cplx E$ in $\cdbc{Z}$ is said to be of
\emph{finite homological dimension over $T$}  if $\cplx E\lotimes
\bL f^\ast \cplx G$ is bounded for any $\cplx{G}$ in $\cdbc{T}$.
\end{defin}

Under the assumption that $X\to T$ is locally projective, it is known \cite[Proposition 2.7]{HLS08} that a relative integral functor $\fmf{\cplx K}{X}{Y}$ given by an object $\cplx K$ in $\cdbc{X\times_T Y}$ can be extended to a functor $\fmf{\cplx K}{X}{Y}\colon D(X)\to D(Y)$, mapping $\cdbc{X}$ to $\cdbc{Y}$ if and only if $\cplx K$ is of finite homological dimension over $X$. Moreover if $X\to T$ and $Y\to T$ are projective morphisms and $\fmf{\cplx{K}}{X}{Y}\colon \cdbc{X}\to \cdbc{Y}$ is an equivalence of cate\-gories, then it is also known \cite[Proposition 2.10]{HLS08}  that its kernel $\cplx K\in \cdbc{X\times_T Y}$ has to be of finite homological dimension over both $X$ and $Y$.

A functor $F\colon \cdbc{X}\to \cdbc{Y}$ is said to be $T$-{\it linear} if for any $\cplx{E}\in \cdbc{X}$ and any $\cplx{N}\in \cdbc{T}$, one has
$$F(\cplx{E}\lotimes \bL p^\ast\cplx{N})\simeq F(\cplx{E})\lotimes \bL q^\ast\cplx{N}\, .$$
The projection formula shows that any relative integral functor  $\Phi^\cplx{K}_{X\to Y}$ is $T$-linear. A relative integral functor $\Phi^\cplx{K}_{X\to Y}$ is said to be a {\it relative Fourier-Mukai transform} if it is an equiva\-lence.  From now on we will denote by $\Aut_T\cdbc{X}$ the group of all $T$-linear autoequivalences of $\cdbc{X}$ and by  $\FM_T(\cdbc{X})$ the subgroup consisting of relative Fourier-Mukai transforms, that is, the group of autoequivalences $\Phi^\cplx{K}$ of $\cdbc{X}$ whose kernel is an object, not just of $\cdbc{X\times X}$, but also of $\cdbc{X\times_T X}$.

\begin{rem} \begin{enumerate}\item As we have mentioned above, when the morphism $X\to T$ is  projective, the kernel of any $\Phi^\cplx{K}\in \FM_T(\cdbc{X})$ has finite homological dimension over $X$.
\item In the case that $X$ is a projective scheme itself, recent results by Ballard \cite{Ballard09} (see also \cite{OL10}), prove that any autoequivalence of $\cdbc{X}$ is a (absolute) Fourier-Mukai transform, so that we are studying the subgroup of $\Aut\cdbc{X}$ consisting of those Fourier-Mukai transforms whose kernel is supported on the fibred product $X\times_T X$.
\end{enumerate}
\end{rem}

For the sake of simplicity we will fix the following notations. For each closed point $x\in X$, we denote by $j_x\colon\{x\}\times Y_t\hra X\times_T Y$ the inclusion of the fibre $\{x\}\times Y_t=\pi_1^{-1}(x)$, and $i_x\colon \{x\}\times Y_t\hra Y$ denotes the composition $\pi_2\circ j_x$, with $t=p(x)$. Furthermore, for any sheaf $\cK$ on $X\times_T Y$ flat over $X$ we denote by $\cK_x$ the sheaf on $Y_{t}$ obtained by restricting $\cK$ to the fibre $\{x\}\times Y_t$ of $\pi_1\colon X\times_T Y\to X$. If $\cplx K$ is an object in
$\cdbc{X\times_T Y}$ and $\Phi=\fmf{\cplx{K}}{X}{Y}$, for any
closed point $t\in B$ we denote by $\Phi_t\colon \dmc{X_t} \to \dmc{Y_t}$ the integral
functor defined by $\bL j_t^\ast \cplx{K}$, with
$j_t\colon X_t\times Y_t\hookrightarrow X\times_T Y$  the closed immersion of the fibre of $\rho\colon X\times_T Y\to T$ over $t\in T$.

Assume that the morphisms  $p\colon X\to T$ and $q\colon Y\to T$ are both {\it proper and flat}.
From the base change formula (see \cite{BBH08}) we obtain that
\begin{equation}\label{e:basechange}\bL
i_t^\ast\Phi (\cplx{F})\simeq\Phi_t(\bL i_t^\ast\cplx{F})\, ,
\end{equation}
for every $\cplx{F}\in \catD(X )$, where $i_t\colon
X_t\hookrightarrow X$ and $i_t\colon Y_t\hookrightarrow Y$ are the
natural embeddings.  In this situation, the base change formula also
gives that
\begin{equation}\label{e:directimage}
i_{t\ast}\Phi_t(\cplx{G})\simeq\Phi ( i_{t\ast}\cplx{G})\, ,
\end{equation}
for every $\cplx{G}\in \catD(X_s)$.
Using this formula, the following result (see \cite[Proposition 2.15]{HLS08}) proves, in great generality, that a relative
integral functor is a Fourier-Mukai transform if and only
if the absolute integral functors induced on the fibres are Fourier-Mukai transforms.

\begin{prop}\label{prop:relativeequiv}
Let $X\to T$ and $Y\to T$ be  proper and flat morphisms. Assume that $X\to T$ is locally projective and let $\cplx{K}$ be
an object in $\cdbc{X\times_T Y}$ of finite homological dimension
over both $X$ and $Y$. The relative integral functor
$\Phi=\fmf{\cplx{K}}{X}{Y}\colon \cdbc{X}\to \cdbc{Y}$ is an equivalence if and only if $\Phi_t\colon
\cdbc{X_t}\to \cdbc
{Y_t}$ is an equivalence
for every closed point $t\in T$.\qed
\end{prop}

Under the conditions of the above proposition, we consider for any morphism $Z\to T$  the following base change diagram
$$
\xymatrix{ & Z\times_T X\times_T Y\ar[dr]^{\pi_{ZY}}\ar[dl]_{\pi_{ZX}} &\\
                   Z\times_T X\ar[dr]_{\pi_1} & & Z\times_T Y\ar[dl]^{\pi_1}\\
                   & Z &}
$$
where the morphisms are the natural projections. The complex $\mathcal{K}_Z^{\scriptscriptstyle\bullet}=\pi_{XY}^\ast\cplx K$ gives rise to a relative integral functor $\Phi_Z\colon D^-(Z\times_T X)\to D^-(Z\times_T Y)$ given by the formula
\begin{equation}\label{eq:basechangeFM}
\Phi_Z(\cplx E)=\bR{\pi_{ZY}}_\ast(\pi_{ZX}^\ast \cplx E\lotimes \rcplx{K}{Z})\,.
\end{equation}
If the kernel $\cplx K$ is of finite homological dimension over $X$, then $\rcplx{K}{Z}$ is of finite homological dimension over $Z\times_T X$ and therefore $\Phi_Z$ maps $\cdbc{Z\times_T X}$ to $\cdbc{Z\times_T Y}$. 
A straightforward consequence of Proposition \ref{prop:relativeequiv} is the following result.
\begin{cor} Let $X\to T$ and $Y\to T$ be  proper and flat morphisms. Assume that $X\to T$ is locally projective and let $\cplx{K}$ be
an object in $\cdbc{X\times_T Y}$ of finite homological dimension
over both $X$ and $Y$
 such that the relative integral functor $\fmf{\cplx K}{X}{Y}\colon\cdbc{X}\to\cdbc{Y}$ is an equivalence of categories. Then, for every morphism $Z\to T$ the relative integral functor $\Phi_Z\colon \cdbc{Z\times_T X}\to\cdbc{Z\times_T Y}$ is also an equivalence of categories.\qed
\end{cor}

We shall be often interested in studying cases in which an integral functor applied to a complex or a sheaf yields a concentrated complex.

\begin{defin}
Given a (relative) integral functor $\Phi=\fmf{\cplx K}{X}{Y}$, an object $\cplx E$ on $D^-(X)$ is said to be WIT$_i$-$\Phi$, if the $j$-th cohomology sheaf $\Phi^j(\cplx E)=0$ for all $j\neq i$. Equivalently $\cplx E$ is WIT$_i$-$\Phi$ if there is a sheaf $\widehat{\cplx E}$ on $Y$ such that $\Phi(\cplx E)\simeq\widehat{\cplx E}[-i]$.  In this case, the sheaf $\widehat{\cplx E}$ is called the Fourier-Mukai transform of $\cplx E$.
\end{defin}

\begin{prop}\label{prop:WITopen}
Let $X\to T$ be a flat morphism and $\cE$ be a sheaf on $X$ flat over $T$. 
\begin{enumerate}
\item The restriction $\cE_t$ to the fibre  $X_t$ is WIT$_i$-$\Phi_t$ for every $t\in T$ if and only if  $\cE$ is WIT$_i$-$\Phi$ and $\widehat\cE$ is flat over $T$. Moreover, restriction is compatible with base change, that is, $(\widehat\cE)_t\simeq\widehat{\cE_t}$ for every $t\in T$.
\item The set $U$ of points in $T$ such that the restriction $\cE_t$ of $\cE$ to the fibre $X_t$ is WIT$_i$-$\Phi_t$ has a natural structure of open (possibly empty) subscheme of $T$.
\end{enumerate}
\end{prop}
\begin{proof}
See \cite[Corollary 6.3 and Proposition 6.4]{BBH08} 
\end{proof}

By applying Proposition \ref{prop:WITopen} to the integral functor $\Phi_X\colon \cdbc{X\times_T X}\to \cdbc{X\times_T Y}$ and to the structure sheaf $\cO_\Delta$ of the relative diagonal in $X\times_T X$ and taking into account Equation \eqref{e:directimage}, one has the following fact which we shall need later  (see \cite[Proposition 6.5]{BBH08}).
\begin{cor}\label{prop:skyscraper}
Let $p\colon X\to T$ be a flat morphism. The set $U$ of points $x$ in $X$ such that the skyscraper sheaf $\cO_x$ is WIT$_i$-$\Phi$ has a natural structure of open subscheme of $X$.\qed
\end{cor}

Recall that an object $\cplx G$ in $\cdbc{X}$ is said to be {\it perfect} if it is isomorphic to a bounded complex of locally free sheaves of finite rank.
\begin{lem}\label{lem:perfectFM}
Assume that $X\to T$ is a locally projective morphism. Let $\cplx{K}\in\cdbc{X\times_T Y}$ be of finite homological dimension over $X$  and  $\Phi=\fmf{\cplx{K}}{X}{Y}\colon \cdbc{X}\to \cdbc{Y}$ the corresponding integral functor. If the kernel $\cplx{K}$ is of finite homological dimension over $Y$, then $\fmf{\cplx{K}}{X}{Y}$ maps perfect objects of $\cdbc{X}$ to perfect objects of $\cdbc{Y}$. If the morphism $X\to T$ is projective, the converse is also true.
\end{lem}

\begin{proof} Suppose that $\cplx{K}$ is of finite homological dimension over $Y$ and take  $\cplx{F}$ a perfect object in $\cdbc{X}$. By \cite[Lemma 1.2]{HLS07}, we have to prove that $\bR\dSHom{\cO_{Y}} (\Phi(\cplx{F}), \cplx{G})$ is in $\dbc{Y}$ for
every $\cplx{G}\in\dbc{Y}$. Indeed, one has that 
$$\begin{aligned}
\bR\dSHom{\cO_{Y}} (\Phi(\cplx{F}), \cplx{G})& \simeq 
\bR\pi_{_2\ast}\bR\dSHom{\cO_{X\times_TY}} (\pi_1^\ast\cplx{F}\lotimes \cplx{K}, \pi_2^!\cplx{G})\simeq \\ &
\simeq \bR\pi_{_2\ast}\bR\dSHom{\cO_{X\times_TY}} (\pi_1^\ast\cplx{F}, \bR\dSHom{\cO_{X\times_TY}} (\cplx{K}, \pi_2^!\cplx{G}))\, ,
\end{aligned} $$ 
where the first isomorphism is relative Grothendieck  duality and  the second is  \cite[Theorem A]{Spal98}. Since $\cplx{K}$ is of finite homological dimension over $Y$ and $X\to T$ is locally projective,   by \cite[Proposition 2.3]{HLS08} $\bR\dSHom{\cO_{X\times_TY}} (\cplx{K},\pi_2^!\cplx{G})$ is a bounded complex. Thus one concludes from \cite[Lemma 1.2]{HLS07} since $\pi_1^\ast\cplx{F} \in \cdbc{X\times_TY}$ is a perfect object.

Suppose now that $X\to T$ is projective and let us prove the converse.  By \cite[Proposition 2.3]{HLS08}, we have to prove that $\bR\pi_{2\ast}(\cplx{K}(r))$ is a  perfect complex for any $r\in \mathbb{Z}$ where $\cplx{K}(r)=\cplx{K}\otimes \pi_1^{\ast}\cO_X(r)$. This is immediate because $\bR\pi_{2\ast}(\cplx{K}(r))\simeq \bR\pi_{2\ast}(\cplx{K}\otimes \pi_1^\ast \cO_X(r))=\fmf{\cplx{K}}{X}{Y}(\cO_X(r))$.
\end{proof}

\begin{cor} If  $X\to T$ is a projective morphism, any relative Fourier-Mukai transform $\Phi^{\cplx{K}}\in \FM_T(\cdbc{X})$  sends perfect objects to perfect objects. \qed
\end{cor}

Let $p_i\colon X_i\to T$ and $q_i\colon Y_i\to T$ be proper flat morphisms and let $\cplx{K_i}\in \cdbc{X_i\times_T\ Y_i})$  be two objects, $i=1,2$. The exterior tensor product $$\cplx{K_1}\boxtimes\cplx{K_2}=\pi_{X_1Y_1}^\ast \cplx{K_1}\lotimes \pi_{X_2Y_2}^\ast \cplx{K_2},$$ defines a relative integral functor
$\Phi^{\cplx{K_1}\boxtimes\cplx{K_2}}\colon D^-(X_1\times_T X_2)\to D^-(Y_1\times_T Y_2)$ denoted also by $\Phi^{\cplx{K_1}}\times \Phi^{\cplx{K_2}}$.

\begin{prop} \label{p:exterior}If  $p_i\colon X_i\to T$ and $q_i\colon Y_i\to T$ for $i=1,2$ are projective and $\Phi^{\cplx{K_i}}\colon \cdbc{X_i}\to\cdbc{Y_i}$ are equivalences, then the exterior tensor product defines  an equivalence $\Phi^{\cplx{K_1}\boxtimes\cplx{K_2}}\colon \cdbc{X_1\times_T X_2}\to\cdbc{Y_1\times_T Y_2}$.
\end{prop}

\begin{proof} Since we are in the projective case, $\cplx{K_1}$ has finite homological dimension over $X_1$.  Then $\pi_{X_1Y_1}^\ast (\cplx{K_1})$ has finite homological dimension over $X_1\times_T X_2\times_T Y_2$ and, since all projections are flat morphisms,  also  over $X_2\times_T Y_2$ and   $X_1\times_TX_2$. Being of finite homological dimension over $X_2\times_T Y_2$ implies that $\cplx{K_1}\boxtimes\cplx{K_2}$ is a bounded complex. To see that $\cplx{K_1}\boxtimes\cplx{K_2}$ has finite homological dimension over $X_1\times_TX_2$, we have to prove that  $(\cplx{K_1}\boxtimes\cplx{K_2})\lotimes\pi_{X_1X_2}^\ast \cplx{G}$ is bounded for any bounded complex $\cplx{G}\in \cdbc{X_1\times_TX_2}$. One has   $$(\cplx{K_1}\boxtimes\cplx{K_2})\lotimes\pi_{X_1X_2}^\ast \cplx{G}\simeq \pi_{X_1Y_1}^\ast \cplx{K_1}\lotimes \pi_{X_1X_2Y_2}^\ast(h^\ast\cplx{K_2}\lotimes r^\ast \cplx{G})\, ,$$ where $h\colon X_1\times_T X_2\times_T Y_2\to X_2\times_TY_2$ and $r\colon X_1\times_T X_2\times_T Y_2\to X_1\times_TX_2$ are the projections.  Since $\pi_{X_1Y_1}^\ast (\cplx{K_1})$ is of finite homological dimension over $X_1\times_T X_2\times_T Y_2$, it is enough to check that $\pi_{X_1X_2Y_2}^\ast(h^\ast\cplx{K_2}\lotimes r^\ast \cplx{G})$  is a bounded complex. This follows  from the fact that $\cplx{K_2}$ has finite homological dimension over $X_2$ and then $h^\ast \cplx{K_2}$ has finite homological dimension over $X_1\times_TX_2$. A similar argument proves that  $\cplx{K_1}\boxtimes\cplx{K_2}$  has finite homological dimension over $Y_1\times_TY_2$. 

Hence, to prove that $\Phi^{\cplx{K_1}\boxtimes\cplx{K_2}}$ is an equivalence, by Proposition \ref{prop:relativeequiv} it suffices to prove that for every $t\in T$, the induced functor $\Phi^{({\cplx{K_1}\boxtimes\cplx{K_2}})_t}$ between the derived categories of the fibres is an equivalence. By the base change formula we have that $\Phi^{({\cplx{K_1}\boxtimes\cplx{K_2}})_t}=\Phi^{{\cplx{K_1}_t\boxtimes\cplx{K_2}}_t}$ and then we are reduced to the absolute case. For smooth fibres this is Assertion 1.7 in \cite{Or02}. Moreover, the proof by Orlov is still valid for arbitrary fibres as long as the integral functor has a right adjoint functor which is again an integral functor. This is true because the kernels have finite homological dimension over both factors, see \cite[Proposition 2.9]{HLS08}.
\end{proof}

To finish this section let us state the following key results that will be used all along the paper.
\begin{prop} \label{prop: ncte} Let $\cplx{K}$ be an object in $\cdbc{X\times_ T Y}$. Suppose $X$ is connected and that for every closed point $t\in T$, the restriction $\bL j_t^\ast \cplx{K}\simeq \mathcal{K}_t[n_t]$ where $\cK_t$ is a sheaf on $X_t\times Y_t$ flat over $X_t$ and $n_t\in \mathbb{Z}$. Then, $\cplx{K}\simeq \mathcal{K}[n]$ for some  sheaf $\mathcal{K}$ on $X\times_T Y$ flat over $X$ and  some $n\in \mathbb{Z}$.
\end{prop} 

\begin{proof} Let $\Phi=\fmf{\cplx{K}}{X}{Y}$ be the relative integral functor of kernel $\cplx{K}$. If $x\in X$ is a closed point lying over $t\in T$ we have, following the above notations, that $ \Phi(\mathcal{O}_x)\simeq {i_x}_\ast\bL j_x^\ast\cplx K\simeq {i_x}_\ast (\bL j_{x,t}^\ast \bL j_t^\ast \cplx{K})$ where $j_{x,t}\colon \{x\}\times Y_t\hookrightarrow X_t\times Y_t$ is the inclusion of the fibre of $X_t\times Y_t\to X_t$ on $x\in X_t$ . Since $\cK_t$ is a sheaf on $X_t\times Y_t$ flat over $X_t$, $ \Phi(\mathcal{O}_x)\simeq  {i_x}_\ast ( j_{x,t}^\ast \cK_t[n_t])$. In other words, the skyscraper sheaf $\cO_x$ is WIT$_{-n_t}$-$\Phi$, with $p(x)=t$. Since $X$ is connected, Corollary  \ref{prop:skyscraper} implies that the integer numbers $n_t$ are the same for all $t\in T$,  say $n$. Therefore the restriction $\bL j_x^\ast\cplx K[-n]$ is a concentrated complex $\cK_x$ for every point $x\in X$.  By \cite[Lemma 4.3]{Bri99} $\cplx K[-n]$ is concentrated as well and is a sheaf $\cK$ on $X\times _T Y$ flat over $X$.
\end{proof}

In the following sections we apply this result to study relative Fourier-Mukai transforms of integral elliptic fibrations, abelian schemes and Fano and anti-Fano fibrations. In the case of integral elliptic fibrations, we prove the flatness condition on the sheaves $\mathcal{K}_t$ by showing that they are shifted universal sheaves, in the second case we prove it  by using certain results on semihomogeneous sheaves  and it is immediate for Fano and anti-Fano fibrations.

We finish this section with the following result which characterizes trivial autoequivalences and whose proof is similar to the absolute case one, see for instance \cite[Corollary 1.12]{BBH08}.

\begin{prop}\label{p:trivial} Let $\cplx{K}$ be an object in $\cdbc{X\times_ T Y}$ of finite homological dimension over $X$ and let $\fmf{\cplx{K}}{X}{Y}$ be the corresponding integral functor. Assume that for every closed point $x\in X$ one has $\fmf{\cplx{K}}{X}{Y}(\mathcal{O}_x)\simeq \mathcal{O}_y$ for some closed point $y\in Y$. Then, there exists a  relative morphism $f\colon X\to Y$ and a line bundle $\mathcal{L}$ on $X$ such that $\fmf{\cplx{K}}{X}{Y}(\cplx{E})\simeq \bR f_\ast(\mathcal{L}\otimes\cplx{E})$ for any object $\cplx{E}$ in $\cdbc{X}$.
\end{prop}

\subsection{Fourier-Mukai transforms acting on the Grothendieck group.}\label{subsec:Grothendieck}

\quad
  
\quad

Let $X$ be  a proper scheme over $\Bbbk$,. Let $K(\cdbc{X})$ be the Grothendieck group of the triangulated category $\cdbc{X}$, that is,  the quotient of the free Abelian group generated by the objects of $\cdbc{X}$ modulo expressions coming from distinguished triangles.

It is known  (\cite{Gro77}) that $K(\cdbc{X})\simeq K
(Coh(X))$, and this group is usually denoted by $K_\bullet(X)$.

Given two object s $\cplx{E}, \cplx{F}$ of $\cdbc{X}$ such that   either $\cplx{E}$ or $\cplx{F}$ is perfect, the Euler characteristic is defined as the integer number given by
$$\chi(\cplx E,\cplx F)=\sum_i (-1)^i \dim \Hom_{\cdbc{X}}(\cplx E,\cplx{F}[i])\,.$$

To pass from $\cdbc{X}$ to $K_\bullet(X)$ one considers the map $[\quad]\colon \cdbc{X}\to K_\bullet(X)$ given by $\cplx{F}\mapsto [\cplx F]=\sum(-1)^i[\mathcal{H}^i(\cplx F)]$. Assume now that $K_\bullet(X)$ is generated by perfect objects. 
Then one has the Euler form  $$e_X\colon K_\bullet(X)\times K_\bullet(X)\to \bbZ\, ,$$  defined by $e_X([\cplx E ], [\cplx F])=\chi (\cplx E,\cplx F)$ if either $\cplx{E}$ or $\cplx{F}$ is perfect and extended by linearity. If, in addition, $X$ is Gorenstein with trivial dualizing sheaf, then  by Serre duality $e_X$ is a skew symmetric or symmetric form. Let us denote by $rad\, e_X$ the radical of $e_X$. Then $e_X$ induces a non-degenerate form on $\widetilde{K}_\bullet(X)=K_\bullet(X)/rad\, e_X$ which we denote by $\tilde{e}_X$. 

Any integral functor $\Phi=\fmf{\cplx K}{X}{X}\colon \cdbc{X}\to \cdbc{X}$ induces a group morphism $\phi\colon K_\bullet(X)\to K_\bullet(X)$ that commutes with the projections $[\quad]\colon \cdbc{X}\to K_\bullet(X)$, that is, such that the following diagram commutes 
$$
\xymatrix{
\cdbc{X}\ar[r]^\Phi\ar[d]_{[\quad]} &\cdbc{X}\ar[d] ^{[\quad]} \\
K_\bullet(X)\ar[r]^{\phi}& K_\bullet(X).}
$$
This morphism is given by $\phi(\alpha)={\pi_2}_!\big(\pi_1^\ast\alpha\otimes [\cplx K]\big)$. Note that $\phi$ is well defined because we are assuming that $K_\bullet(X)$ is generated by perfect objects. Moreover if $\Phi$ is an equivalence then $\phi$ is an isomorphism.

Hence the group of Fourier-Mukai transforms $\FM(\cdbc{X})$ acts on $K_\bullet(X)$ by automorphisms. When $X$ is projective, since any Fourier-Mukai transform maps perfect objects to perfect objects (Lemma \ref{lem:perfectFM}) and for any Fourier-Mukai transform $\Phi$ one has $\chi(\cplx E,\cplx F)=\chi(\Phi(\cplx E),\Phi(\cplx F))$,  this action preserves the Euler form $e_X$ and  $rad\, e_X$. Therefore any Fourier-Mukai transform $ \Phi\colon \cdbc{X}\isom \cdbc{X}$ defines an isomorphism on $\widetilde K_\bullet(X)$ which preserves the non-degenerate  form $\tilde e_X$.
\section{Relative Fourier-Mukai transforms for Weiertra\ss\ fibrations}
In this section we study  the group of relative Fourier-Mukai transforms of  the derived category  of any relatively integral elliptic fibration over a connected base scheme. For the convenience of the reader, we start by briefly recalling the structure of the group of derived equivalences of the fibres of these fibrations.
\subsection{Derived equivalences for integral curves of genus one}

\quad

\quad

Let $C$ be a connected integral projective curve of arithmetic genus one. Then $C$ is either a smooth elliptic curve, or a singular rational curve with a simple node or cusp. 

Let $K_\bullet(C)$ be the Grothendieck group of $C$. We distinguish two cases:
\begin{enumerate}
\item
If $C$ is smooth, then $\big(\rk,\det\big)\colon K_\bullet(C)\isom \bbZ\oplus \Pic(C)$ (see for instance \cite[Chapter II Ex 6.11]{Hart77}).
\item If $C$ is singular, then $\big(\rk,\deg\big )\colon K_\bullet(C)\isom \bbZ\oplus\bbZ$ with generators $[\cO_x]$ and $[\cO_C]$, being $x\in X$  a smooth point, which are both perfect objects (\cite[Lemma 3.1]{BuKr05}). 
\end{enumerate}

In both cases,  the Euler form $e_C\colon K_\bullet(C)\times K_\bullet(C)\to \bbZ$ is well defined and one has $(\rk, \deg)\colon \widetilde K_\bullet(C)\isom \bbZ\oplus\bbZ\,$ with generators $[\cO_C]$  and $[\cO_x]$ for any smooth point $x\in C$.

By Subsection \ref{subsec:Grothendieck} any Fourier-Mukai transform $\Phi=\fmf{\cplx K}{C}{C}\colon \cdbc{C}\to \cdbc{C}$ induces an isomorphism on $\widetilde K_\bullet(C)\simeq\bbZ\oplus\bbZ$ which preserves the symplectic form $\tilde e_C$ on $\widetilde K_\bullet(C)$. 
So there exists a group morphism $\Aut\cdbc{C}\xrightarrow{\ch} \SL_2(\bbZ)$ such that if $\cplx F$ is an object in $\cdbc{C}$, then 
$$\begin{pmatrix}
\rk(\Phi(\cplx F))\\
\deg(\Phi(\cplx F))
\end{pmatrix}=\ch(\Phi)
\begin{pmatrix}
\rk(\cplx F)\\
\deg( \cplx{F})
\end{pmatrix}.
$$
Note that in order to determine the matrix $\ch(\Phi)$ it is sufficient to compute the rank and degree of $\Phi(\cO_C)$ and $\Phi(\cO_x)$, with  $x$ a non-singular point of $C$.  

\begin{exe}\label{exe:generators}
\begin{enumerate}

\item Let $f\colon C\to C$ be an automorphism of $C$, $\calL$ a line bundle on $C$ and $n$ an integer number. Denote by $\Phi_0$ the integral functor of kernel ${\Gamma_f}_\ast\calL[n]$, being $\Gamma_f\colon C \hra C\times C$ the graph of $f$. Then  $\Phi_0$ is an equivalence isomorphic to $f_\ast(-\otimes\calL)[n]$ and $\ch(\Phi_0)=(-1)^n\begin{pmatrix}1 & 0 \\ \deg\calL & 1\end{pmatrix}$.

\item Let $\delta\colon C \hra C\times C$ be the diagonal immersion of $C$ and $\cI_\Delta$ its ideal sheaf. Denote by $\Phi_1$ the integral functor of kernel $\cI_\Delta$. Then $\Phi_1$ is a Fourier-Mukai transform and $\ch(\Phi_1)=\begin{pmatrix} -1 & 1 \\ 0 & -1\end{pmatrix}$ (see for instance \cite[Propositions 1.3 and 1.9]{HLST09}). 

\item Let $x_0$ be a non-singular point of $C$ and $\Phi_2=-\otimes_{\cO_C}\cO_C(x_0)$. Then $\Phi_2$ is a Fourier-Mukai transform and $\ch(\Phi_2)=\begin{pmatrix} 1 & 0 \\ 1 & 1\end{pmatrix}$.
\end{enumerate}
\end{exe}

The existence of the Fourier-Mukai transforms $\Phi_1$ and $\Phi_2$ in  Example \ref{exe:generators} allows one to prove that the morphism  $\Aut\cdbc{C}\xrightarrow{\ch} \SL_2(\bbZ)$ is surjective because they are mapped to a pair of generators for $\SL_2(\bbZ)$. 

The following well known result gives a description of the group of Fourier-Mukai transforms of $\cdbc{C}$.  If  $C$ is a smooth elliptic curve, a proof can be found in \cite{HiVdB}. For a rational curve with a node or a cusp this was proved in \cite{BuKr05}.
\begin{thm}\label{thm:sequencefibrewise}
Let $C$ be a connected integral projective curve with arithmetic genus one.   The following exact sequence holds
\begin{equation}\label{eq:sequencefibrewise}
1\lra \Aut^0\cdbc{C}\lra \Aut\cdbc{C}\stackrel{\ch}\lra \SL_2(\bbZ)\lra 1\, ,
\end{equation}
where  $\Aut^0\cdbc{C}=\Aut (C)\rtimes (2\bbZ \times \Pic^0(C))$ is the subgroup of $\Aut \cdbc{C}$ consisting of autoequivalences of the form $f_\ast(-\otimes_{\mathcal{O}_C}\mathcal{L})[n]$, with $\mathcal{L}\in \Pic^0(C)$, $f\in \Aut(C)$ and $n\in 2\bbZ$.
\end{thm}
\qed

\begin{rema}\label{rema:WITsemistable}
Since $\Aut\cdbc{C}\stackrel{\ch}\lra \SL_2(\bbZ)$ is surjective one has that the group of Fourier-Mukai transforms $\Aut\cdbc{C}$, modulo $\Aut^0\cdbc{C}$, is generated by $\Phi_1$ and $\Phi_2$, in the notation of Example \ref{exe:generators}. By \cite[Proposition 1.13 and Theorem 1.20]{HLST09} any semistable sheaf on $C$ is WIT with respect to both transforms and both transforms preserve the semistability condition. On the other hand it is clear that any element of $\Aut^0\cdbc{C}$ maps sheaves into sheaves and preserves semistability. Therefore any semistable sheaf is  WIT-$\Phi$, with respect to any $\Phi\in \Aut\cdbc{C}$ .
\end{rema}

The following result is possibly known, but we include a proof for completeness since we have not been able to provide a reference for it.
\begin{prop}\label{prop:flatkernel} Let $C$ be a  connected integral projective curve of arithmetic genus one. 
If $\Phi^\cplx{K}\in \Aut \cdbc{C}$, then $\cplx{K}$ is isomorphic, up to shift, to a sheaf $\cK$ on $C\times C$, flat over both factors.
\end{prop}          
\begin{proof}
Let $\ch(\Phi^\cplx{K})=
\begin{pmatrix}
 c & a \\
 d &b
\end{pmatrix}\in \SL_2(\bbZ)\,
$, then $a$ and $b$ are coprime numbers. We can assume $a\geq 0$, otherwise we take the Fourier-Mukai transform given by $\cplx K[1]$ instead of $\cplx K$. We consider the moduli space $\cM_C(a,b)$ of semistable sheaves on $C$ with rank $a$ and degree $b$. Since $a$ and $b$ are coprime numbers $\cM_C(a,b)$  is a fine moduli space parametrizing stable sheaves with such topological invariants, so that there is a universal sheaf $\cP$ on $\cM_C(a,b)\times C$. On the other hand from \cite{BBH08} one has isomorphisms
$$\cM_C(a,b)\simeq\cM_C(1,0)\simeq\cM_C(0,1)\simeq C.$$
By the isomorphisms $\cM_C(a,b)\times C\iso \cM_C(1,0)\times C\iso C\times C$ and up to twisting by pullbacks of line bundles on $\cM_C(1,0)$,  we get that $\cP\simeq \cI_\Delta\otimes{\pi_2}^\ast\cO_C(x_0)$, where $x_0$ is a non-singular point on $C$. This implies that $\Phi^{\cP}$ is a Fourier-Mukai transform and $$\ch(\Phi^{\cP})=\begin{pmatrix}
 c+at& a \\
 d+bt&b
\end{pmatrix},$$  for some $t\in \mathbb{Z}$.
Changing $\mathcal{P}$ by its twist with the pull-back of a line bundle on $C$ of degree $-t$, we get that $\ch(\Phi^\cP)=\ch(\Phi^\cplx{K})$. Then,  the two transforms $\Phi^\cP$ and $\Phi^\cplx{K}$ are isomorphic up to an object of $\Aut^0\cdbc{C}$, so that there exists a sheaf $\cK$ on $C\times C$ flat over both factors and an integer number $n$ such that $\cplx{K} \simeq\cK[n]$.
\end{proof}

\begin{exe} The following example shows that Proposition \ref{prop:flatkernel} is no longer true for reducible curves of arithmetic genus one. Consider $p\colon S\to T$ a smooth elliptic surface with some reducible not multiple fibre. That is,  the generic fibre of $p$ is isomorphic to a smooth elliptic curve and for some point  $0\in T$ of the discriminant, the fibre $S_0:=p^{-1}(0)$ is one of the reducible not multiple Kodaira fibres (see the list in \cite{kod} for the possible types). Assume that $p$ is projective.

For any $t\in T$, let us denote by $j_t\colon S_t\hookrightarrow S$ the immersion of the fibre.  Let  $C$ be any irreducible component of $S_0$ and $i\colon C\hookrightarrow S$ the natural inclusion. Since $C$ is  a $(-2)$-curve in $S$, then $E:=i_\ast \mathcal{O}_{C}$ is a spherical object of the derived category $\cdbc S$ of the surface.  Seidel and Thomas proved in \cite{SeTh01} that the corresponding twist functor $T_E$ is an autoequivalence of $\cdbc S$.  It is a general fact that this twist functor is equal to the Fourier-Mukai transform $\Phi:=\Phi^{\cplx{I(E)}}$ where the kernel $\cplx{I(E)}\in \cdbc{S\times S}$  is given by 
\begin{equation}\label{e:kertwist}
\cplx{I(E)}=\text{cone}(E^\vee\boxtimes E\to \mathcal{O}_\Delta)\, .
\end{equation}

Since $\Phi$ is an equivalence of categories, it follows from  \cite[Proposition 2.10]{HLS08} that  $\cplx{I(E)}$ has finite homological dimension over both factors. Moreover, $\cplx{I(E)}$ belongs to $\cdbc{S\times_T S}$, that is, $\Phi$ is a relative Fourier-Mukai transform. Hence, using Proposition \ref{prop:relativeequiv}, one has that for any $t\in T$, the restriction $\Phi_t$ is an autoequivalence of  the derived category $\cdbc{S_t}$ of the fibre. For $t\neq 0$, this is just the identity functor, but for the singular fibre $S_0$, it is a non-trivial autoequivalence $\Phi_0$. Let us show that  the kernel $\bL j_0^\ast \cplx{I(E)}$ of $\Phi_0$ is a genuine complex with at least two non-zero cohomology sheaves.

Using the exact sequence $$0\to \mathcal{O}_S(-C)\to \mathcal{O}_S\to E\to 0\, ,$$ one gets that  ${\bL j_0^\ast E}$ is isomorphic to a genuine complex $\cplx{E}$ 
\begin{equation}\label{e:resolucion}
\cplx{E}:=\{{\mathcal{O}_S(-C)}|_{S_0}\xrightarrow{f}\mathcal{O}_{S_0}\}\, ,
\end{equation}
with two non-zero cohomology sheaves, $\mathcal{H}^{-1}(\cplx{E})=\mathcal{T}or^1_{\mathcal{O}_S}(\mathcal{O}_C,\mathcal{O}_{S_0})$ and $\mathcal{H}^0(\cplx{E})={j_0}^\ast (E)$.

  The base change formula shows that the kernel $\bL j_0^\ast \cplx{I(E)}$ is isomorphic to \begin{equation}\label{e:triangle}
\text{cone}(\cplx{E}^\vee\boxtimes \cplx{E}\to \mathcal{O}_{\Delta_{0}}) \, ,
\end{equation} where $\Delta_0$ is the diagonal of $S_0$.

Consider the spectral sequence
$$E_2^{p,q}=\mathcal{E}xt^p(\mathcal{H}^{-q}(\cplx{E}),\mathcal{O}_{S_0})=\mathcal{H}^p(\mathcal{H}^{-q}(\cplx{E})^\vee)\implies \mathcal{H}^{p+q}(\cplx{E}^\vee)\, .$$
Since $\mathcal{H}^{-1}(\cplx{E})$ and $\mathcal{H}^{0}(\cplx{E})$  are both of pure dimension 1 and for any pure dimension 1 sheaf $\mathcal{E}xt^i(\mathcal{F}, \mathcal{O}_{S_0})=0 $ for all $i\neq0$, one has that the complex $\cplx{E}^\vee$  is given by
$$\cplx{E}^\vee=\{\mathcal{O}_{S_0}\xrightarrow{f^\ast}(\mathcal{O}_S(-C)|_{S_0})^\ast\}\, ,$$ and it
has non-trivial cohomology in degree 0 and 1.

The complex $\cplx{E}$ is concentrated in $[-1,0]$ with $\mathcal{H}^0(\cplx{E})\neq 0$ and $\cplx{E}^\vee$ is concentrated in $[0,1]$ and $\mathcal{H}^1(\cplx{E}^\vee)\neq 0$, then using the spectral sequence that computes the cohomology sheaves of the tensor product of two complexes
$$\mathcal{T}or_{-p}(\mathcal{H}^q(\cplx{F}), \cplx{G})\implies \mathcal{T}or_{-(p+q)}(\cplx{F}, \cplx{G})\, ,$$ one obtains that
$\mathcal{H}^{1}(\cplx{E}^\vee\boxtimes \cplx{E})\neq 0$.

On the other hand, if $d$ denotes the differential of the complex $\cplx{E}^\vee\boxtimes \cplx{E}$, we have $$d^{-1}=d_{\pi_1^\ast \cplx{E}^\vee}^0\otimes 1+1\otimes d_{\pi_2^\ast \cplx{E}}^{-1}=\pi_1^\ast f^\ast\otimes 1+1\otimes \pi_2^\ast f\, .$$ Since $\ker f$ and $\ker f^\ast$  are both non-zero, one has $\mathcal{H}^{-1}(\cplx{E}^\vee\boxtimes \cplx{E})=\ker d^{-1}\neq 0$.

Looking at the long exact sequence of cohomology corresponding to the triangle defined by \eqref{e:triangle},  one obtains that $\mathcal{H}^{-i}(\bL j_0^\ast \cplx{I(E)})\neq 0$ for $i=0,2$ which proves our claim.
\end{exe}

\subsection{Relative Fourier-Mukai transforms and Weiertra\ss\ fibrations}

\quad

\quad

Let $p\colon X\to T$ be a relatively integral elliptic fibration, that is, a proper flat morphism of schemes  whose fibres are integral  Gorenstein curves of arithmetic genus one. Generic fibres of $p$ are smooth elliptic curves, and the degenerated fibres are rational curves with one node or one cusp.
We will assume that the base scheme $T$ is connected and that there exists a regular section $\sigma\colon T\hookrightarrow X$ of $p$, that is, the image of $\sigma$ does not contain any singular point of the fibres.

Relatively integral elliptic fibrations have a Weierstra\ss\ form \cite[Lemma II.4.3]{Mir89}: If we denote by $H=\sigma(T)$ the image of the section and $\omega=R^1p_\ast\cO_X\simeq (p_\ast \omega_{X/T})^{-1}$, the line
bundle $\cO_X(3H)$ is relatively very ample and if 
$\cE=p_\ast \cO_X(3H)\simeq\cO_T\oplus\omega^{\otimes
2}\oplus\omega^{\otimes3}$ and $\bar p\colon \bbP(\cE^\ast)=\operatorname{Proj}(S^\bullet(\cE))\to T$
is the associated projective bundle, there is a closed immersion $j\colon X\hookrightarrow \bbP(\cE^\ast)$
of $T$-schemes such that $j^\ast\cO_{\bbP(\cE^\ast)}(1)=\cO_X(3H)$. In particular, $p$ is a projective morphism.

\begin{prop}Let $p\colon X\to T$ be a relatively integral elliptic fibration and let $\Phi=\Phi^{\cplx{K}}\colon \cdbc{X}\isom \cdbc{X}$ be a relative Fourier-Mukai transform. Then $\cplx{K}$ has only one non-trivial cohomology sheaf, that is, $\cplx{K}\simeq \mathcal{K}[n]$ for some coherent sheaf $\mathcal{K}$ on $X\times_T X$  and some  $n\in \mathbb{Z}$. Moreover, the sheaf $\mathcal{K}$ is flat over both factors.
\end{prop}
\begin{proof} By Proposition \ref{prop:relativeequiv}, all the absolute integral functors  $\Phi_t\colon \cdbc{X_t}\to \cdbc{X_t}$ induced on the fibres  are  equivalences due to the projectivity of the morphism $p$.  Thus, Proposition \ref{prop:flatkernel} proves that for every  $t\in T$ one has $\bL j_t^\ast \cplx{K}\simeq \mathcal{K}_t[n_t]$ where $\mathcal{K}_t$ is a sheaf on $X_t\times X_t$ flat over both factors. Hence, one concludes  by Proposition  \ref{prop: ncte}.
\end{proof}

For every point $t\in T$ there exists a morphism $\rho_t\colon \FM_T(\cdbc{X})\to \Aut \cdbc{X_t}$ where $\rho_t(\Phi^\cplx{K})$ is the Fourier-Mukai transform $\Phi_t$ defined by $\bL{j_t}^\ast\cplx K$.

\begin{prop}
There is a surjective morphism $\FM_T(\cdbc{X})\stackrel{\widetilde{\ch}}\lra \SL_2(\bbZ)$ such that for any point $t\in T$ the diagram 
$$\xymatrix{
\FM_T(\cdbc{X})\ar[r]^(.58){\widetilde{\ch}}\ar[d]_{\rho_t} & \SL_2(\bbZ)\\
\Aut \cdbc{X_t}\ar[ur]_{\ch_{X_t}}&
}$$
is commutative.
\end{prop}

\begin{proof}
Let $\Phi\in \FM_T(\cdbc{X})$ be a  relative Fourier-Mukai transform. Let us define $$\widetilde{\ch}(\Phi):=\ch_{X_t}(\Phi_t)\, ,$$ for any $t\in T$, where $\ch_{X_t}$ is  the group morphism for the curve $X_t$  given by Equation (\ref{eq:sequencefibrewise}).

 Let us  prove that  $\widetilde{\ch}$ is well defined.
As we have already remarked, to obtain the matrix $\ch_{X_t}(\Phi_t)$ it is enough to compute the rank and degree of $\Phi_t(\cO_{X_t})$ and $\Phi_t(\cO_{x})$ with $ x$ a non-singular point of $X_t$.

Let us consider the relative Fourier-Mukai transform $\Phi_X\colon\cdbc{X\times_T X}\to\cdbc{X\times_T X}$ given by Equation (\ref{eq:basechangeFM}) and take, on the one hand,  
 the structure sheaf $\cO_{X\times_TX}$. For every $x\in X$ such that $p(x)=t$, its restriction  $\mathcal{O}_{X_t}$ to the fibre  $X_t=p^{-1}(t)=\pi_1^{-1}(x)$ is, by Remark \ref{rema:WITsemistable}, WIT$_i$-$\Phi_t$ for some $i$. Since $\mathcal{O}_{X\times_T X}$  is a sheaf  flat over both factors we get, by Proposition \ref{prop:WITopen}, that $\cO_{X\times_T X}$ is WIT$_i$-$\Phi_X$, its Fourier-Mukai transform $\widehat{\cO_{X\times_TX}}$ is flat over both factors of $X\times_T X$ and  is compatible with base change. That is, for every $x\in X_t$, we have   $(\widehat{\cO_{X\times_TX}})_x[-i] \simeq \Phi_t({\cO_{X_t}})$ .  Thus, the rank and degree of $\Phi_t(\cO_{X_t})$ do not  depend on $t$. 
 
Consider on the other hand the structure sheaf $\cO_\Delta$ of the relative diagonal immersion $\delta\colon X\hra X\times_T X$. Then $\Phi_X(\cO_\Delta)\simeq \cplx{K}$ where $\cplx{K}$ is the kernel  of the original  relative Fourier-Mukai transform $\Phi$. As we have seen before, $\cplx{K}\simeq \mathcal{K}[n]$ for some  sheaf $\mathcal{K}$ on $X\times_S X$ flat over $X$ and  some $n\in \mathbb{Z}$. Again, base change compatibility  gives that  $\mathcal{K}_x[n]\simeq\Phi_t(\mathcal{O}_{x})$ for every $x\in X$ with $p(x)=t$. Thus  the rank and the degree of $\Phi_t(\mathcal{O}_{x})$ don't depend on $t$ either. Therefore, the morphism $\widetilde{\ch}$ is well defined.

To conclude that $\FM_T(\cdbc{X})\stackrel{\widetilde{\ch}}\lra \SL_2(\bbZ)$ is a surjective morphism, one has just to consider the relative Fourier-Mukai transforms given by $\cI_\Delta$ and $\delta_\ast\cO_X(\sigma(T))$, where $\cI_\Delta$ is the ideal sheaf of the relative diagonal immersion of $X$. By Example \ref{exe:generators} the matrices corresponding to these equivalences are $ \begin{pmatrix} -1 & 1\\ 0 &-1 \end{pmatrix}$ and $ \begin{pmatrix} 1 & 0\\ 1 & 1 \end{pmatrix}$ respectively. Since these matrices are a pair of generators for the group $\SL_2(\bbZ)$ the result follows.
\end{proof}

\begin{thm}\label{t:FMW}
There exists an  exact sequence
$$1\lra \Aut^0_T\cdbc{X}\lra \FM_T(\cdbc{X})\stackrel{\widetilde{\ch}}\lra \SL_2(\bbZ)\lra 1\, ,$$
with $\Aut^0_T\cdbc{X}=\Aut(X/T)\rtimes(2\bbZ\times \Pic^{\underline 0}(X)) $, where $$\Pic^{\underline 0}(X)=\{\calL\in \Pic(X) ;  \deg(\calL_t)=0 \text{, for any } t\in T\}.$$
\end{thm}

\begin{proof}
It only remains to compute the kernel of $\widetilde{\ch}$. Let $\Phi=\Phi^{\cplx K}\in \FM_T(\cdbc{X})$ be a relative Fourier-Mukai transform. We know that  $\cplx{K}\simeq \mathcal{K}[n]$ where $\cK$ is a sheaf on $X\times_T X$ flat over $X$ and $n\in \mathbb{Z}$. For any $x\in X$, we have  $\Phi(\cO_x)\simeq{i_x}_\ast\cK_x[n]$ and, if $p(x)=t$, then $\cK_x[n]\simeq \Phi_t(\mathcal{O}_x)$. Suppose that  $\Phi$ is in the kernel of $\widetilde{\ch}$. This means that  $\deg (\Phi_t(\cO_x))=1$ and $\rk(\Phi_t(\cO_x))=0$ for any $t\in T$ and any $x\in X_t$. Since $\cK_x$ is a sheaf, this implies that ${i_x}_\ast\cK_x$ is a skyscraper sheaf $\cO_y$ and $n\in2\bbZ$. Moreover as $\Phi$ is a  relative Fourier-Mukai transform, $x$ and $y$ belong to the same fibre. Then $\Phi[-n]$ sends skyscraper sheaves to skyscraper sheaves. By Proposition \ref{p:trivial} one gets that $\Phi^{\cplx K}\simeq \bR f_\ast(-\otimes\calL)[n]$, with $\calL$ a line bundle on $X$ and $f\colon X\to X$ a relative morphism. The fact that $\Phi$ is a relative equivalence implies that $f\in \Aut(X/T)$. Thus, $\bR f_\ast= f_\ast$. Finally, since $\Phi$ lies in the kernel of $\widetilde{\ch}$, we deduce that $\calL\in \Pic^{\underline 0}(X)$.
\end{proof}

\begin{rem}\label{r:Pic}Notice that since $p\colon X\to T$ has a regular section $\sigma\colon T\to X$, the exact sequence
$$0\to \Pic(T)\xrightarrow{p^\ast} \Pic(X)\to \textbf{Pic}_{X/T}(T)\to 0\, ,$$ splits via the retraction defined by  $\sigma \in X(T)$. In particular, the exact sequence
$$0\to \Pic(T)\xrightarrow{p^\ast}\Pic^{\underline 0}(X)\to \textbf{Pic}^0_{X/T}(T)\to 0\, ,$$
splits, that is, $\text{Pic}^{\underline 0}(X)\simeq \textbf{Pic}^0_{X/T}(T)\times \Pic(T)$. 
\end{rem}

\section{Relative Fourier-Mukai transforms for abelian schemes}
\subsection{Derived equivalences for abelian varieties}
\quad

\quad

 A complete study that includes a geometric interpretation of any exact equivalence between the derived categories of coherent sheaves $\cdbc{A}$ and $\cdbc{B}$ of two abelian varieties $A$ and $B$ was carried out by Polishchuk \cite{Pol96} and Orlov \cite{Or02}. The latter  also gave a  description of all exact autoequivalences of the derived category of coherent sheaves of an abelian variety over an algebraically closed field.

 Let $A$ be an abelian variety over $\Bbbk$. Denote $\hat{A}$ the dual abelian variety and $\mathcal{P}$ the Poincar\'e line bundle on $A\times \hat{A}$.
 
Following  Mukai \cite{Muk78}, for a coherent sheaf $\mathcal{F}$ on $A$ we consider the subgroup
$$\Upsilon(\mathcal{F})=\{ (a,\alpha)\in A\times \hat{A}\text{ such that } T_a^\ast \mathcal{F}\simeq \mathcal{F}\otimes \mathcal{P}_\alpha \} \subset A\times \hat{A}\, ,$$ where $T_a\colon A\to A$ denotes the translation by $a\in A$.
The sheaf $\mathcal{F}$ is said to be {\it semihomogeneous} if $\dim \Upsilon(\mathcal{F})=\dim A$.

The following result is due to Orlov \cite[Proposition 3.2]{Or02}.

\begin{prop} \label{VarAb}Let $A, B$ be abelian varieties, and let $\cplx{K}$ be an object of
$\cdbc{A \times B}$ such that the integral functor $\Phi^\cplx{K}\colon \cdbc{A}\isom \cdbc{B}$  is a Fourier-Mukai transform.
Then $\cplx{K}$ has only one non-trivial cohomology sheaf, that is, $\cplx{K}\simeq \mathcal{K}[n]$  where $\mathcal{K}$ is a sheaf on $A \times B$   and $n\in \mathbb{Z}$.  
\end{prop}

We can make more precise the nature of the kernel of an equivalence between abelian varieties.

\begin{prop} \label{flatness} The sheaf $\mathcal{K}$ associated to an equivalence $\Phi^{\cplx{K}}\colon \cdbc{A}\isom  \cdbc{B}$ is a semihomogeneous sheaf and it
is flat over both factors.
\end{prop}

\begin{proof} 
Since $\Phi^{\cplx{K}}$ is an equivalence, by \cite[Theorem 2.10 and Corollary 2.13]{Or02} there is an isomorphism
$$f_{\cplx{K}}\colon A\times \hat{A}\to B\times \hat{B}\, ,$$ and $f_{\cplx{K}}(a,\alpha)=(b,\beta)$ if and only if 

\begin{equation}\label{e:traslacion}
T_{(a,b)}^\ast \mathcal{K}\otimes \pi_1^\ast \mathcal{P}_\alpha\simeq\mathcal{K}\otimes \pi_2^\ast \mathcal{P}_\beta\, ,
\end{equation}
where $\pi_i$ for $i=1,2$ are the natural projections of $A\times B$ onto its factors. This means that $(a,\alpha,b,\beta)\in \Gamma_{f_\cplx{K}}$ if and only if $(a,b,\alpha^{-1},\beta)\in \Upsilon(\mathcal{K})$. Since the dimension of the graph $\Gamma_{f_\cplx{K}}$ is equal to $2g$, one has that $\dim\Upsilon(\mathcal{K})=2g$ and thus $\mathcal{K}$ is a semihomogeneous sheaf on $A\times B$.

 Let us see that $\mathcal{K}$ is flat for $\pi_1\colon A\times B\to A$. 
By  generic flatness  \cite[Theorem 6.9.1]{EGAIV2},  there exists an open subset $U\subset A$ such that $\mathcal{K}|_{U\times B}$ is flat over $U$.  Since $U$ is a non-empty open subset of an abelian variety, for any $a\notin U$, one has that $a=x_1+x_2$ with $x_1, x_2\in U$.  Then, translating $U$ and using again  Equation \eqref{e:traslacion} we obtain that $\mathcal{K}$ is flat everywhere.

The proof for $\pi_2\colon A\times B\to B$ is the same.
\end{proof}
Any isomorphism $f\colon A\times \hat{A}\isom B\times\hat{B}$ can be written as a matrix  $\begin{pmatrix} \alpha &\beta\\
 \gamma & \delta
\end{pmatrix}$  where $\alpha\colon A\to B$, $\beta\colon \hat{A}\to B$, $\gamma\colon A\to \hat{B}$ and $\delta\colon \hat{A}\to \hat{B}$ are morphisms of abelian varieties. One defines the isomorphism $$f^\dag\colon B\times \hat{B}\isom A\times\hat{A}\, ,$$ given by the matrix $\begin{pmatrix} \hat{\delta}&-\hat{\beta}\\
-\hat{\gamma} & \hat{\alpha}
\end{pmatrix}$.
 We denote by $U(A\times\hat{A}, B\times\hat{B})$ the subgroup of isomorphisms $f\colon A\times\hat{A}\isom B\times \hat{B}$ that are isometric, that is, such that $f^\dag=f^{-1}$. When $A=B$, it is denoted by $U(A\times\hat{A})$.

The following result \cite[Theorem 4.14]{Or02} gives a complete description of the group of autoequivalences of the derived category $\cdbc{A}$ of an  abelian variety $A$ over an algebraically closed field $\Bbbk$.
\begin{thm}\label{thm:sequenceav} Let $A$ be an abelian variety over $\Bbbk$. Then one has the following exact sequence of groups $$0\to \mathbb{Z}\oplus(A\times \hat{A})\to \Aut\cdbc{A}\to U(A\times\hat{A})\to 1\, ,$$ where the autoequivalence defined by
 $(n, a, L)\in \mathbb{Z}\oplus (A\times \hat{A})$ is $$\Phi^{(n, a, L)}(\cplx{E})=T_{a\ast}(\cplx{E})\otimes L[n]\, .$$
\end{thm}

 \subsection{Relative Fourier-Mukai transforms and abelian schemes}
 
 \quad
 
 \quad
 
In this section we are interested in studying the group of  relative Fourier-Mukai transforms of an abelian scheme. 
Let $p\colon X\to T$ be an abelian scheme over a scheme $T$, that is, a smooth proper  morphism with connected fibres such that there exist morphisms of $T$-schemes $$m_X\colon X\times_T X \to X,\quad\quad i\colon X\to X,\quad\quad e\colon T\to X,$$ corresponding respectively to the group law, inversion and unit section and satisfying the usual group relations. Consider $\hat{p}\colon \hat{X}=\Pic^0_{X/T}\to T $ the dual abelian scheme, whose closed points correspond to line bundles whose scheme-theoretic support is contained in some fibre of $p$ and belong to the connected component of the identity of the Picard group of that fibre. There is a Poincar\'e line bundle $\mathcal{P}$ over $X\times_T \hat{X}$  that we normalize by imposing that its restriction to ${e(T)\times_T \hat{X}}$ is trivial.

By \cite[Theorem 1.1]{Muk87b}, the relative integral functor  defined by the Poincar\'e bundle $$\Phi^{\mathcal{P}}\colon  \cdbc{\hat{X}}\to \cdbc{X}\, ,$$ is a Fourier-Mukai transform whose right adjoint $\Phi^{\mathcal{P}}_{\mathcal{R}}$ is the integral functor with kernel 
$\mathcal{P}^\ast \otimes \pi_1^\ast \omega_{\hat{X}/T}[g]$, where $\pi_1\colon \hat{X}\times_T X\to \hat{X}$  is the natural projection  and $g$ is the relative dimension of $p\colon X\to T$.

A coherent sheaf $\mathcal{F}$ on $X$ is said to be {\it relatively semihomogeneous} if it is flat over $T$ and for any $t\in T$ its restriction $\mathcal{F}_t$ to the fibre $X_t$ is a semihomogeneous sheaf.

From now on we  assume that the base scheme $T$ is connected.

\begin{prop} \label{t:haz}Let $p\colon X\to T$ and $q\colon Y\to T$ be two projective abelian schemes, and let $\Phi=\Phi^{\cplx{K}}\colon \cdbc{X}\isom \cdbc{Y}$ be a relative Fourier-Mukai transform. Then $\cplx{K}$ has only one non-trivial cohomology sheaf, that is, $\cplx{K}\simeq \mathcal{K}[n]$ for some coherent sheaf $\mathcal{K}$ on $X\times_T Y$  and some  $n\in \mathbb{Z}$. Moreover, the sheaf $\mathcal{K}$ is flat over $X$ and relatively semihomogenous over $T$.
\end{prop}
\begin{proof}
By Proposition \ref{prop:relativeequiv}, all the absolute integral functors  $\Phi_t\colon \cdbc{X_t}\to \cdbc{Y_t}$ induced on the fibres  are  equivalences because of the projectivity of the morphisms.  Thus, Propositions \ref{VarAb} and \ref{flatness} prove that for every  $t\in T$ one has $\bL j_t^\ast \cplx{K}\simeq \mathcal{K}_t[n_t]$, where $\mathcal{K}_t$ is a sheaf on $X_t\times Y_t$ flat over both factors and semihomogeneous. Since we are assuming that $T$ is connected, it follows that $X$ is connected as well and we finish by Proposition \ref{prop: ncte}.
\end{proof}

Let $p\colon X\to T$ and $q\colon Y\to T$ be two projective abelian schemes.  Proposition \ref{p:exterior} allows us to generalise Definition 9.34 in \cite{Huybbook} to the case of abelian schemes.
\begin{defin} \label{d:equivprod}To any relative Fourier-Mukai transform $\Phi^{\cplx{K}}\colon \cdbc{X}\to \cdbc{Y}$ we associate the relative Fourier-Mukai transform $$\Psi^{\cplx{K}}\colon \cdbc{X\times_T \hat{X}}\to\cdbc{Y\times_T\hat{Y}}\, ,$$ defined as the composition

\begin{equation*}
    \xymatrix@C=4.4pc{
        \cdbc{X\times_T\hat{X}} \ar[r]^{\Psi^{\cplx{K}}} \ar[d]_{id\times \Phi^{\mathcal{P}_X}} &  \cdbc{Y\times_T\hat{Y}} \\
        \cdbc{X\times_T X} \ar[d]_{\mu_{X\ast}}  &  \cdbc{Y\times_TY}\ar[u]_{(id\times \Phi^{\mathcal{P}_Y})^{-1}}\\
        \cdbc{X\times_T X} \ar[r]^{\Phi^{\cplx{K}}\times \Phi^{\cplx{K}}_\mathcal{R}} & \cdbc{Y\times_TY}\,,\ar[u]_{\mu_{Y\ast}}
        }
    \end{equation*}
where $\mathcal{P}_X$ and $\mathcal{P}_Y$ are the Poincar\'e bundles for $X$ and $Y$ respectively, $\mu_X$ is the relative automorphism
$$\begin{aligned}\mu_X=m_X\times Id\colon & X\times_T X\to X\times_T X\\
&(x_1,x_2)\mapsto (x_1+x_2,x_2)\, ,
\end{aligned}$$ and $\Phi^{\cplx{K}}_\mathcal{R}$ is the right adjoint to $\Phi^{\cplx{K}}$ considered as a functor from $\cdbc{X}$ to $\cdbc{Y}$.

\end{defin}

\begin{lem} The construction $\Phi^{\cplx{K}}\to \Psi^{\cplx{K}}$ is compatible with  composition, that is, given $X, Y,$ and $ Z$ projective abelian schemes over $T$ and two relative Fourier-Mukai transforms,
$$\Phi^{\cplx{K}}\colon \cdbc{X}\isom\cdbc{Y}, \quad\quad \Phi^{\cplx{R}}\colon\cdbc{Y}\isom\cdbc{Z},$$  then one has that $\Psi^{\cplx{R}\ast_T\cplx{K}}\simeq \Psi^{\cplx{R}}\circ\Psi^{\cplx{K}}$.
\end{lem}

\begin{proof} Since $\Psi$ is defined as  composition of relative Fourier-Mukai transforms, the statement is obtained directly from
$$(\Phi^{\cplx{R_1}}\times \Phi^{\cplx{R_2}})\circ(\Phi^{\cplx{K_1}}\times \Phi^{\cplx{K_2}})\simeq \Phi^{\cplx{R_1}\ast_T\cplx{K_1}}\times \Phi^{\cplx{R_2}\ast_T\cplx{K_2}}\, ,$$ which follows from Equation \eqref{e:compos} describing the kernel of the composition of two relative integral functors.
\end{proof}
Since  $\Phi^{\cplx{K}}=Id$ implies $\Psi^{\cplx{K}}=Id$, from the above Lemma we get the following result.
\begin{cor} The map  
$$\begin{aligned} \Psi\colon  \FM_T(\cdbc{X}) &\to \FM_T(\cdbc{X\times_T\hat{X}})\\
\Phi^{\cplx{K}}&\mapsto \Psi^{\cplx{K}}\, ,
\end{aligned}$$ is a group morphism.\qed
\end{cor}

\begin{prop}\label{p:equivprod} Let $\Phi^{\cplx{K}}\colon \cdbc{X}\to \cdbc{Y}$ be a relative Fourier-Mukai transform between the derived categories of two projective abelian schemes $X$ and $Y$ over $T$. Then, the relative Fourier-Mukai transform $$\Psi^{\cplx{K}}\colon \cdbc{X\times_T\hat{X}}\to \cdbc{Y\times_T\hat{Y}},$$ is given by $$\Psi^{\cplx{K}}\simeq(\mathcal{L}_{\cplx{K}}\otimes (-))\circ f_{\cplx{K}\ast}\, ,$$ where $\mathcal{L}_{\cplx{K}}\in \Pic(Y\times_T\hat{Y})$ and $f_{\cplx{K}}\colon X\times_T\hat{X}\isom Y\times_T \hat{Y}$ is an isomorphism of abelian schemes over $T$.
\end{prop}

\begin{proof} 
Let $t\in T$ be a closed point and  let $$\Psi^{\cplx{K}_t}\colon \cdbc{X_t\times \hat{X}_t}\isom \cdbc{Y_t\times \hat{Y}_t}\, , $$ be  the equivalence associated to $\Phi^{\cplx{K}_t}\colon \cdbc{X_t}\to \cdbc{Y_t}$ as in Definition \ref{d:equivprod} (see also Definition 9.34 in \cite{Huybbook}). By the very definition of $\Psi$ one finds that $(\Psi^{\cplx{K}})_t\simeq\Psi^{\cplx{K}_t}$. By \cite[Theorem 2.10]{Or02}, we have that $(\Psi^{\cplx{K}})_t$ sends skyscraper sheaves to skyscraper sheaves.  Since $\Psi^{\cplx{K}}$ is a relative Fourier-Mukai transform, from its compatibility with direct images given in Equation \eqref{e:directimage}, we deduce that $\Psi^{\cplx{K}}$ also sends skyscraper sheaves to skyscraper sheaves. Then using Proposition \ref{p:trivial}, one obtains that $$\Psi^{\cplx{K}}\simeq(\mathcal{L}_{\cplx{K}}\otimes (-))\circ f_{\cplx{K}\ast}\, .$$  Since $\Psi^{\cplx{K}}$ is an equivalence, we conclude that $f_{\cplx{K}}$ is an isomorphism. Moreover, since $$(\Psi^{\cplx{K}})_t\simeq (\mathcal{L}_\cplx{K}|_{Y_t\times \hat{Y}_t}\otimes (-))\circ f_{\cplx{K}\ast}|_{X_t\times \hat{X}_t}\, ,$$ and by  \cite[Theorem 2.10]{Or02} we know that $f_{\cplx{K}}|_{X_t\times \hat{X}_t}\colon X_t\times \hat{X}_t\to Y_t\times \hat{Y}_t$ is a morphism of abelian varieties for any closed point $t\in T$, one concludes that $f_{\cplx{K}}\colon X\times_T\hat{X}\isom Y\times_T \hat{Y}$ is an isomorphism of abelian schemes over $T$.
\end{proof}

\begin{cor} If $X\to T$ is a projective abelian scheme, the map 
$$\begin{aligned} \gamma_X\colon \FM_T(\cdbc{X}) &\to \Aut_T(X\times_T\hat{X})\\
\Phi^{\cplx{K}}&\mapsto f_{\cplx{K}} \, ,
\end{aligned}$$ is a group morphism.\qed
\end{cor}

We can associate to any relative morphism $f\colon X\times_T\hat{X}\to Y\times_T \hat{Y}$ a matrix
 $\begin{pmatrix} \alpha &\beta\\
 \gamma & \delta
\end{pmatrix}$ 
where $\alpha\colon X\to Y$, $\beta\colon \hat{X}\to Y$, $\gamma\colon X\to \hat{Y}$ and $\delta\colon \hat{X}\to \hat{Y}$ are morphisms over $T$. If moreover $f$ is an isomorphism, it determines another isomorphism
$$f^\dag\colon Y\times_T \hat{Y}\isom X\times_T\hat{X}\, ,$$ whose matrix is $\begin{pmatrix} \hat{\delta}&-\hat{\beta}\\
-\hat{\gamma} & \hat{\alpha}
\end{pmatrix}$.
Following Mukai  \cite{Muk78}, Polishchuk \cite{Pol96} and Orlov \cite{Or02} we consider the following: 
\begin{defin}\label{d:isometric} An isomorphism of abelian schemes $f\colon X\times_T\hat{X}\isom Y\times_T \hat{Y}$ over $T$ is said to be  isometric if $f^\dag=f^{-1}$. We denote by $U(X\times_T\hat{X}, Y\times_T\hat{Y})$ the subgroup of isometric isomorphisms $f\colon X\times_T\hat{X}\isom Y\times_T \hat{Y}$. When $X=Y$, it is denoted by $U(X\times_T\hat{X})$.

\end{defin} 

The following result is a consequence of the Rigidity Lemma \cite[Corollary 6.2]{GIT}, see Proposition 3.5.1 in \cite{BL10}.

\begin{lem}\label{l:rigid} Let $X\to T$, $Y\to T$ be abelian schemes over a connected base $T$. The restriction map $\Hom_T(X,Y)\to \Hom(X_t, Y_t)$ is injective for any $t\in T$.
\end{lem}

\begin{prop} If the base scheme $T$ is connected, the isomorphism $$f_{\cplx{K}}\colon X\times_T\hat{X}\isom Y\times_T\hat{Y}$$ corresponding to any relative Fourier-Mukai transform $\Phi^{\cplx{K}}\colon \cdbc{X}\to \cdbc{Y}$  is isometric.
\end{prop}
\begin{proof} For any closed point $t\in T$, we know that 
$$(f_{\cplx{K}})_t\simeq f_{\cplx{K}_t}\colon X_t\times \hat{X}_t\isom Y_t\times\hat{Y}_t\, ,$$ is the isomorphism associated to the equivalence $\Phi^{\cplx{K}_t}\colon \cdbc{X_t}\isom \cdbc{Y_t}$. By \cite[Proposition 2.18]{Or02}, $(f_{\cplx{K}})_t$ is isometric, thus
$$(f_{\cplx{K}})_t^\dag(f_{\cplx{K}})_t=Id_{X_t\times\hat{X}_t}=(Id_{X\times_T\hat{X}})_t\quad \text{ and }\quad (f_{\cplx{K}})_t(f_{\cplx{K}})_t^\dag=Id_{Y_t\times\hat{Y}_t}=(Id_{Y\times_T\hat{Y}})_t\, ,$$ and we conclude by Lemma \ref{l:rigid} .
\end{proof}

\begin{cor} For any projective abelian scheme $X\to T$ over a connected base $T$, one has a group morphism $\gamma_X\colon \FM_T(\cdbc{X})\to U(X\times_T\hat{X})$.	\qed
\end{cor}

\begin{defin}
We say that two abelian schemes $p\colon X\to T$,  $q\colon Y\to T$ are \textit{relative Fourier-Mukai partners} if the bounded derived categories of coherent sheaves $\cdbc{X}$ and $\cdbc{Y}$ are relative Fourier-Mukai equivalent, that is if there exists a relative Fourier-Mukai transform $\Phi\colon\cdbc{X}\xrightarrow{\sim}\cdbc{Y}$.
\end{defin}

As a consequence of the above we obtain the following result:

\begin{thm}\label{t:FM-partners} Let $X\to T$, $Y\to T$ be projective abelian schemes over a connected base $T$. If $X$ and $Y$ are relative Fourier-Mukai partners, then there exists an isometric isomorphism between $X\times_T\hat{X}$ and $Y\times_T\hat{Y}$ over $T$.\qed
\end{thm}

\begin{prop} Let $p\colon X\to T$ be a projective abelian scheme with $T$ connected. The kernel of the morphism $$\gamma_X\colon \FM_T(\cdbc{X})\to U(X\times_T\hat{X})\, ,$$ is isomorphic to the group $\mathbb{Z}\oplus (X(T)\times \hat{X}(T)\times \Pic(T))$ where $X(T)$ and $\hat{X}(T)$ denote the groups of sections of $X$ and $\hat{X}$, respectively. The autoequivalence defined by $(n, x, L, M)\in \mathbb{Z}\oplus (X(T)\times \hat{X}(T)\times \Pic(T))$ is $$\Phi^{(n, x, L, M)}(\cplx{E})=T_{x\ast}(\cplx{E})\otimes L\otimes p^\ast M[n]\, .$$
\end{prop}

\begin{proof} Let $\Phi=\Phi^{\cplx{K}}\in \FM_T(\cdbc{X})$  be a  relative Fourier-Mukai transform. For every closed point $t\in T$, the restriction $\Phi_t$ is a Fourier-Mukai transform, and then, using Proposition \ref{t:haz}, we know that $\cplx{K}\simeq \mathcal{K}[n]$ where $\mathcal{K}$ is a sheaf on $X\times_TX$ flat over both factors and $n\in \mathbb{Z}$. Suppose that $\Phi$ is in the kernel of $\gamma_X$. Thus, $f_{\cplx{K}_t}=Id_{X_t\times \hat{X}_t}$ for any $t\in T$. By \cite[Proposition 3.3]{Or02} that describes the kernel of  the analogous morphism in the absolute case, the equivalences $\Phi^{\cplx{K}_t}$ transform skyscraper sheaves into skyscraper sheaves  up to shift. One sees that $\Phi^{\cplx{K}}$ has the same property by using Equation \eqref{e:directimage} and  the same argument as in the proof of Proposition \ref{p:equivprod}. Hence, Proposition \ref{p:trivial} 
proves that  $\Phi^{\cplx{K}}\simeq f_\ast (-)\otimes \mathcal{L}[n]$ where $f\colon X\to X$ is a relative automorphism and $\mathcal{L}\in \Pic(X)$.
Thus for any $t\in T$ 
$$\Phi^{\cplx{K}_t}\simeq f|_{X_t\ast}(-)\otimes \mathcal{L}|_{X_t}[n]\, .$$ Using again \cite[Proposition 3.3]{Or02} we get that $f|_{X_t}\simeq T_{x_t}$ is a translation by some $x_t\in X_t$ and $\mathcal{L}|_{X_t}\in \Pic^0(X_t)$, that is, $\mathcal{L}\in\Pic^{\underline 0}(X)$. Then, there is a section $x\in X(T)$ so that $f\simeq T_x$. 

By using the unit section $e\colon T\hookrightarrow X$  and proceeding as in Remark \ref{r:Pic}, we have $\Pic^{\underline 0}(X)\simeq \hat X(T)\times \Pic(T)$. Therefore, for any   $\mathcal{L}\in\Pic^{\underline 0}(X)$ there exist $L\in\hat{X}(T)$ and $M\in \Pic(T)$ satisfying $\mathcal{L}\simeq L\otimes p^\ast M$.
To finish the proof it is enough to check that the group law in $\mathbb{Z}\oplus (X(T)\times \hat{X}(T)\times \Pic(T))$ is the direct product law. This follows from the commutativity of shifts with translations and tensor products by line bundles and to the fact that  for every $x\in X(T)$, $L\in \hat{X}(T)$ and $M\in \Pic(T)$ one has $T_x^\ast(L\otimes p^\ast M)\simeq L\otimes p^\ast M$.\end{proof}

\begin{cor}\label{c:FMabeliano} For any projective abelian scheme $X\to T$ where $T$ is connected, one has a group exact sequence
$$0\to \mathbb{Z}\oplus (X(T)\times \hat{X}(T)\times \Pic(T))\to \FM_T(\cdbc{X})\to U(X\times_T\hat{X})\, .$$
\end{cor}

\begin{rem} Notice that a projective abelian scheme $X\to T$ of relative dimension 1 can be considered as a Weierstra\ss\ fibrations without singular fibres. Hence,  Theorem \ref{t:FMW} applies in this situation and thus there exists an exact sequence
$$1\lra \Aut^0_T\cdbc{X}\lra \FM_T(\cdbc{X})\stackrel{\widetilde{\ch}}\lra \SL_2(\bbZ)\lra 1\, ,$$
with $$\Aut^0_T\cdbc{X}=\Aut(X/T)\rtimes(2\bbZ\times \Pic^{\underline 0}(X))=\Aut_{abel}(X/T)\rtimes(2\bbZ\times X(T)\times \hat{X}(T)\times \Pic(T)),$$ where $\Aut_{abel}(X/T)$ denotes the abelian scheme automorphisms.
\end{rem}

\begin{thm}\label{t:abelpol} Let $X\to T$ be an abelian scheme over a connected base such that  there exists a line bundle $L\in\Pic(X)$ inducing a principal polarization and $\End_T(X)=\mathbb{Z}$. Then $U(X\times_T X)\simeq \SL_2(\mathbb Z)$ and there is an exact sequence $$0\to \mathbb{Z}\oplus (X(T)\times \hat{X}(T)\times \Pic(T))\to \FM_T(\cdbc{X})\to \SL_2(\mathbb Z)\to 1.$$
\end{thm}

\begin{proof} If $\End_T(X)=\mathbb Z$ then $\End_T(\hat X)=\mathbb Z$, $\Hom_T(X,\hat X)\simeq \mathbb Z$, and $\Hom(\hat X,X)\simeq \mathbb Z$. Therefore, if $\lambda\colon X\xrightarrow{\sim} \hat X$ is the principal polarization, every element $f\in U(X\times_T \hat X)$ is of the form $f=\begin{pmatrix} a_X & b\cdot\lambda^{-1}\\
c\cdot\lambda & d_{\hat X}
\end{pmatrix}$, where $a, b, c, d\in\mathbb Z$. Definition \ref{d:isometric} implies that $\begin{pmatrix} a &b\\ c & d\end{pmatrix}\in \SL_2(\mathbb Z)$, proving the first claim. Let us consider the relative Fourier-Mukai transforms $(\lambda^{-1})_*\circ\Phi^{\mathcal P}\colon \cdbc{X}\to\cdbc{ X}$ and $L\otimes(-)\colon\cdbc{ X}\to\cdbc{X}$, where $\mathcal P$ is the Poincar\'e line bundle. A straightforward computation shows that their associated symmetric isomorphisms are respectively $\begin{pmatrix} 0 &-\lambda^{-1}\\ \lambda & 0
\end{pmatrix}$, $\begin{pmatrix} 1_X & 0\\ \lambda & 1_{\hat X}
\end{pmatrix}$. These correspond via the identification $U(X\times_T X)\simeq \SL_2(\mathbb Z)$ to a pair of generators, therefore $\gamma_X\colon \FM_T(\cdbc{X})\to U(X\times_T\hat{X})$ is surjective and we are done.
\end{proof}

This is our general theory for abelian schemes over an arbitrary base. In order to obtain further results we need to restrict the type of the base scheme and in particular we endeavour now to study the case where it is normal. 

From now on we suppose that $T$ is normal and connected, and thus integral. We denote by $\eta\in T$ the generic point. We start by recalling Lemma 4.1 in \cite{Pol02}.

\begin{lem}\label{l:rest-fibre}
Let $Z$ and $W$ be abelian schemes over $T$. The restriction map to the generic fibre $$\begin{aligned}\Hom_T(Z, W) &\to \Hom(Z_\eta,W_\eta)\\
f &\mapsto f_\eta\, ,\end{aligned}$$
is an isomorphism. Moreover $f$  is an isogeny if and only $f_\eta$ is an isogeny.
\end{lem}

Given two abelian schemes over $T$, we denote by $U_0(X\times_T\hat{X}, Y\times_T\hat{Y})$ the subset of $U(X\times_T\hat{X}, Y\times_T\hat{Y})$ formed by those  $f= \begin{pmatrix} \alpha &\beta\\
 \gamma & \delta
\end{pmatrix}$  such that  $\beta\colon  \hat{X}\to Y$  is an isogeny.

As an immediate consequence of Lemma \ref{l:rest-fibre} we get.

\begin{prop}\label{p:rest-fibre} Let $X$ and $Y$ be abelian schemes over a normal base $T$. The restriction map gives an isomorphism $U(X\times_T\hat{X}, Y\times_T\hat{Y}) \xrightarrow{\sim} U(X_\eta\times\hat{X}_\eta, Y_\eta\times\hat{Y}_\eta)$  and $f\in U_0(X\times_T\hat{X}, Y\times_T\hat{Y})$ if and only if $f_\eta\in U_0(X_\eta\times\hat{X}_\eta, Y_\eta\times\hat{Y}_\eta)$.
\end{prop}

We also have the following finiteness result.

\begin{thm}\label{t:finiteness} For every abelian scheme $W$ over a normal base $T$ there are only finitely many,  up to isomorphism, abelian schemes that can be embedded in $W$ as abelian sub-schemes.
\end{thm}

\begin{proof} By passing to the generic fibre, every abelian subscheme $Z\subset W$ yields an abelian subvariety $Z_\eta\subset W_\eta$. The result follows now from the finiteness theorem for abelian subvarieties in \cite{LOZ96} and Lemma \ref{l:rest-fibre}.
\end{proof}

Now we can prove a finiteness result for relative Fourier-Mukai partners.

\begin{thm} \label{t:finitepartner}Any abelian scheme $p\colon X\to T$ over a connected normal base has finitely many non-isomorphic relative Fourier-Mukai partners.
\end{thm}

\begin{proof} If $q\colon Y\to T$ is a relative Fourier-Mukai partner of $p\colon X\to T$ then by Theorem \ref{t:FM-partners} there is an isometric isomorphism $f\colon X\times_T\hat X\xrightarrow{\sim} Y\times_T\hat Y$. Therefore $Y$ is an abelian subscheme of $X\times_T\hat X$ and we conclude by Theorem \ref{t:finiteness}.
\end{proof}

For a projective abelian scheme $p\colon X\to T$, let us denote by $\vb^{sh}(X/T)$ (resp. $\vb^{ssh}(X/T)$) the set of relatively semihomogeneous (resp. relatively semihomogeneous and relatively simple) vector bundles on $X$ and let $\NS(X/T)=\Pic(X/T)/\Pic^0(X/T)$ be the relative N\'eron-Severi group. Consider the slope map
$$\begin{aligned} \mu\colon \vb^{sh}(X/T) &\to \NS(X/T)\otimes_\mathbb{Z}\mathbb{Q}\, ,\\
\mathcal{F}&\mapsto \frac{[\det(\mathcal{F})]}{\rk{\mathcal{F}}}\, .
\end{aligned}$$

\begin{prop}\label{t:equivalent} Let $p\colon X\to T$ and $q\colon Y\to T$ be two abelian schemes over  a normal connected base such that the slope map $\mu\colon \vb^{sh}(X\times_T Y/T) \to  \NS(X\times_T Y/T)\otimes_\mathbb{Z}\mathbb{Q}$ is surjective. Then, for any $f\in U(X\times_T\hat{X}, Y\times_T\hat{Y})$ there exists a relative Fourier-Mukai transform $\Phi^{\cplx{K}}$ such that $ f_{\cplx{K}}=f$.
\end{prop}

\begin{proof}  One defines a map $$\xi\colon U_0(X\times_T\hat{X}, Y\times_T\hat{Y})\to \text{Hom}^{sim}_T(X\times_T Y, \hat{X}\times_T \hat{Y})\otimes_\mathbb{Z} \mathbb{Q}\, ,$$ that associates to  $f=\begin{pmatrix} \alpha &\beta\\
 \gamma & \delta
\end{pmatrix}$ the symmetric homomorphism  
$\xi(f)=\begin{pmatrix} \beta^{-1}\alpha &-\beta^{-1}\\
 -{\hat \beta}^{-1} & \delta\beta^{-1}
\end{pmatrix}$. On the other hand, it is well known, see for instance \cite[Proposition 1.2]{DP94} and \cite[Lemme XI 1.6]{Ray70}, that the Mumford map $\varphi\colon \Pic(X\times_T Y/T)\to \Hom_T(X\times_T Y, \hat{X}\times_T \hat{Y})$ induces an isomorphism $$\phi\colon \NS(X\times_T Y/T)\otimes_\mathbb{Z}\mathbb{Q}\simeq \text{Hom}^{sim}_T(X\times_T Y, \hat{X}\times_T \hat{Y})\otimes_\mathbb{Z} \mathbb{Q}\, , $$ given by $\phi(\frac{[\mathcal{L}]}{r})=
\frac{\varphi_{\mathcal{L}}}{r}
$. Therefore, since the slope map is surjective, it follows that for any $f\in U_0(X\times_T\hat{X}, Y\times_T\hat{Y})$ there exists a relatively semihomogeneous and relatively simple  vector bundle $\mathcal{E}$ over $X\times_T Y$ such that $\xi(f)=\phi(\mu(\mathcal{E}))$. Let us consider the integral functor $\Phi^\mathcal{E}\colon\dbc{X}\to \dbc{Y}$. For any $t\in T$ we have $\xi(f_t)=\mu(\mathcal{E}_t)$, $(\Phi^\mathcal{E})_t=\Phi^{\mathcal{E}_t}$ and $\mathcal{E}_t$ is a simple semihomogeneous vector bundle, thus \cite[Proposition 4.11]{Or02} implies that $(\Phi^\mathcal{E})_t$ is an equivalence. Since $T$ is normal, it follows from  \cite[Th\'eor\`eme XI 1.4]{Ray70} that  any abelian scheme over $T$ is projective, therefore we can apply Proposition \ref{prop:relativeequiv} to conclude that $\Phi^\mathcal{E}$ is an equivalence. Moreover since $(f_\mathcal{E})_t=f_{\mathcal{E}_t}$ and by \cite[Proposition 4.12]{Or02} we have $f_{\mathcal{E}_t}=f_t$, we get $(f_\mathcal{E})_t=f_t$. Therefore, by Lemma \ref{l:rigid} we have $f_\mathcal{E}=f$. 

Let us now prove the general case. If $f\notin U_0(X\times_T\hat{X}, Y\times_T\hat{Y})$ then proceeding as in \cite{Or02} page 591, we can write $f_\eta$ as the composition of two maps $f_\eta=g_\eta\circ h_\eta$ where $g_\eta\in U_0(X_\eta\times\hat{X}_\eta, Y_\eta\times\hat{Y}_\eta)$ and $h_\eta\in U_0(X_\eta\times\hat{X}_\eta)$. Therefore, by Proposition \ref{p:rest-fibre}  $f$ is factorized also as the composition of two maps $g\in U_0(X\times_T\hat{X}, Y\times_T\hat{Y})$, $h\in U_0(X\times_T\hat{X})$. Hence we can apply to $g$ and $h$ the previous argument to get the corresponding Fourier-Mukai transforms $\Phi^\mathcal{E}$, $\Phi^\mathcal{F}$ such that $f_\mathcal{E}=g$, $f_\mathcal{F}=h$. Now,  if $\mathcal K^\bullet$ is the convolution of $\mathcal{E}$ and $\mathcal{F}$ then the Fourier-Mukai transform $\Phi^{\mathcal K^\bullet}$ verifies $f_{\mathcal K^\bullet}=f$ and the proof is complete.  \end{proof}

This result and the ones previously proved lead to the following theorems.

\begin{thm} Let $p\colon X\to T$ and $q\colon Y\to T$ be two abelian schemes over  a normal connected base such that the slope map $\mu\colon \vb^{sh}(X\times_T Y/T) \to  \NS(X\times_T Y/T)\otimes_\mathbb{Z}\mathbb{Q}$ is surjective. Then $X$ and $Y$ are relative Fourier-Mukai partners if and only if there is an isometric isomorphism $f\colon X\times_T\hat X\xrightarrow{\sim}Y\times_T\hat Y$.
\end{thm}

\begin{thm}\label{t:FMabelianoNormal} Let $p\colon X\to T$ be an abelian scheme over a connected normal base such that the slope map $\mu\colon \vb^{sh}(X\times_T X/T) \to  \NS(X\times_T X/T)\otimes_\mathbb{Z}\mathbb{Q}$ is surjective. Then there is a short exact sequence of groups:
$$0\to \mathbb{Z}\oplus (X(T)\times \hat{X}(T)\times \Pic(T))\to \FM_T(\cdbc{X})\to U(X\times_T\hat{X})\to 1\, .$$
\end{thm}

\section{Relative Fourier-Mukai transforms for Fano and anti-Fano fibrations}
\subsection{Fourier-Mukai transform for Fano and anti-Fano varieties}

\quad
\quad

Let $X$ be a smooth irreducible projective variety whose anticanonical (Fano case) or canonical (anti-Fano case) sheaf is ample. Under these assumptions Bondal and Orlov proved that $X$ is uniquely determined by its derived category $\cdbc{X}$ and the group of autoequivalences for $\cdbc{X}$ reduces to trivial transforms.  These results have been extended recently to Gorenstein schemes. 

\begin{thm}\label{thm:BO}
Let $X$ be a connected equidimensional Gorenstein projective scheme with ample canonical or anticanonical sheaf, then
\begin{enumerate}
\item If there is an equivalence $\cdbc{X}\simeq\cdbc{Y}$, then $X$ is isomorphic to $Y$.
\item The group of autoequivalences of $\cdbc{X}$ is generated by the shift functor on $\cdbc{X}$, together with pull-backs of automorphisms of $X$ and twists by line bundles, that is
$$\Aut\cdbc{X}\simeq \Aut (X)\ltimes (\Pic(X)\oplus \bbZ)\,.$$
\end{enumerate}
\end{thm}
\begin{proof}
See \cite[Theorem 2.5 and Theorem 3.1]{BO01} for the original arguments in the case of smooth projective varieties. For the singular case see \cite[Corollary 6.3 and Proposition 6.18]{ Ballard11}, see also \cite[Theorem 1.15 and Corollary 1.17]{SS10}.
\end{proof}

Using Proposition \ref{prop: ncte} and the second part of Theorem \ref{thm:BO} one can describe the group of relative Fourier-Mukai transforms for Fano or anti-Fano fibrations.

 \subsection{Relative Fourier-Mukai transforms and Fano or anti-Fano fibrations}
 Let $p\colon X\to T$ be a Fano or anti-Fano fibration, that is, a Gorenstein projective morphism, with T connected, whose fibres have  either ample  (Fano case)  or antiample  (anti-Fano case) canonical sheaf.
 
\begin{thm}\label{t:Fano}
The group $\FM_T(\cdbc{X})$ of relative Fourier-Mukai transforms is generated by relative automorphisms of $X$, twists by line bundles and shifts. That is
$$\FM_T(\cdbc{X})\simeq \Aut (X/T)\ltimes (\Pic(X)\oplus \bbZ)\, .$$
\end{thm}
\begin{proof}
Let $\Phi=\Phi^{\cplx K}\in \FM_T\cdbc{X}$ be a relative Fourier-Mukai transform. For any $t\in T$  one has, by (2) of Theorem \ref{thm:BO}, that $\Phi_t\simeq {f_t}_\ast(\calL_t\otimes -)[n_t]$ with $f_t\in \Aut (X_t)$, \, $\calL_t\in \Pic (X_t)$ and $n_t\in \bbZ$. Hence, the kernel of $\Phi_t$ is isomorphic to $\cK_t[n_t]$,  and $\cK_t$ is a sheaf on $X_t\times X_t$ flat over $X_t$. Since $T$ is connected we obtain by Proposition  \ref{prop: ncte} that $\cplx K\simeq \cK[n]$, with $\cK$  a sheaf on $X\times_T X$ flat over $X$ and $n\in\bbZ$. Thus $\Phi[-n]$ sends skyscraper sheaves to skyscraper sheaves and we conclude by Proposition \ref{p:trivial}.
\end{proof}
\begin{rem} Notice that under the conditions of Corollary 2.12 in \cite{SS10}, this theorem proves that the subgroup $\FM_T(\cdbc{X})$ coincides with the group $\Aut_T \cdbc{X}$ of all exact $T$-linear auto-equivalences of $\cdbc{X}$.
\end{rem}

\begin{rem} In particular, the group of relative Fourier-Mukai transforms of any projective bundle $\mathbb P(\mathcal E)\to T$ coincides with the trivial transforms, that is $$\FM_T(\cdbc{\mathbb P(\mathcal E)})\simeq \Aut (\mathbb P(\mathcal E)/T)\ltimes (\Pic(\mathbb P(\mathcal E))\oplus \bbZ)\, .$$
\end{rem}

\def\cprime{$'$}

\end{document}